\documentclass[a4paper,12pt]{article}
\usepackage{graphicx} 
\usepackage[margin=2cm]{geometry}

\usepackage[T1]{fontenc}
\usepackage[english]{babel}

\usepackage[usenames,dvipsnames,svgnames,pdftex]{xcolor}
\usepackage{amsmath}
\usepackage{amssymb}
\usepackage{amsbsy}
\usepackage{amsthm}
\usepackage{booktabs}
\usepackage[linesnumbered,noend,algoruled]{algorithm2e}
\usepackage{enumitem}
\usepackage[
colorlinks=true,
linkcolor=teal,
citecolor=purple,
]{hyperref}
\usepackage[capitalise,nameinlink]{cleveref}
\usepackage{mathtools}
\usepackage{tikz}
\usetikzlibrary{arrows}
\usetikzlibrary{matrix,decorations.pathreplacing,calc}

\pgfkeys{tikz/mymatrixenv/.style={decoration=brace,every left delimiter/.style={xshift=3pt},every right delimiter/.style={xshift=-3pt}}}
\pgfkeys{tikz/mymatrix/.style={matrix of math nodes,left delimiter=(,right delimiter={)},inner sep=2pt,column sep=1.25em,row sep=1.25em,nodes={inner sep=0pt}}}
\pgfkeys{tikz/mymatrixbrace/.style={decorate,thick}}
\newcommand\mymatrixbraceoffseth{0.7em}
\newcommand\mymatrixbraceoffsetv{0.2em}

\newcommand*\mymatrixbraceleft[4][m]{
    \draw[mymatrixbrace] ($(#1.north west)!(#1-#3-1.south west)!(#1.south west)-(\mymatrixbraceoffseth,0)$)
        -- node[left=2pt] {#4} 
        ($(#1.north west)!(#1-#2-1.north west)!(#1.south west)-(\mymatrixbraceoffseth,0)$);
}
\newcommand*\mymatrixbracetop[4][m]{
    \draw[mymatrixbrace] ($(#1.north west)!(#1-1-#2.north west)!(#1.north east)+(0,\mymatrixbraceoffsetv)$)
        -- node[above=2pt] {#4} 
        ($(#1.north west)!(#1-1-#3.north east)!(#1.north east)+(0,\mymatrixbraceoffsetv)$);
}

\usepackage{amsfonts}

\usepackage{amsrefs} 

\newtheorem{lemma}{Lemma}[section]%
\newtheorem{theorem}[lemma]{Theorem}%
\newtheorem{corollary}[lemma]{Corollary}%
\newtheorem{hypothesis}[lemma]{Hypotheses}%
\newtheorem{conjecture}[lemma]{Conjecture}%

\theoremstyle{definition}
\newtheorem{definition}[lemma]{Definition}%
\newtheorem{example}[lemma]{Example}%

\theoremstyle{remark}
\newtheorem{remark}[lemma]{Remark}%

\SetKw{IOKw}{}

\let\oldnl\nl
\newcommand{\nlnonumber}{\renewcommand{\nl}{\let\nl\oldnl}}

\DontPrintSemicolon

\newcommand{\N}{\mathbb{N}} 
\newcommand{\F}{\mathbb{F}} 

\definecolor{amethyst}{rgb}{0.6, 0.4, 0.8}
\SetKwFor{ForEach}{\textcolor{mycolor}{foreach}}{\textcolor{mycolor}{do}}{\textcolor{mycolor}{end}}%

\definecolor{amethyst}{rgb}{0.59, 0.44, 0.84}
\definecolor{blue(ncs)}{rgb}{0.0, 0.53, 0.74}
\definecolor{dandelion}{rgb}{0.99, 0.76, 0.0}
\SetKwIF{If}{ElseIf}{Else}{\textcolor{amethyst}{if}}{\textcolor{amethyst}{then}}{\textcolor{amethyst}{else if}}{\textcolor{amethyst}{else}}{\textcolor{amethyst}{end if}}%
\SetKwFor{For}{\textcolor{amethyst}{for}}{\textcolor{amethyst}{do}}{\textcolor{amethyst}{endfor}}
\SetKwFor{While}{\textcolor{amethyst}{while}}{\textcolor{amethyst}{do}}{\textcolor{amethyst}{end}}
\SetKw{Return}{\textcolor{blue(ncs)}{return}}%
\SetKwBlock{Begin}{\textcolor{blue(ncs)}{function}}{{end function}}
\SetKwRepeat{Repeat}{\textcolor{amethyst}{repeat}}{\textcolor{amethyst}{until}}

\newcommand{\slp}{\mathfrak{S}} 
\newcommand{\slength}{\Upsilon} 
\newcommand{\ssize}{\mathfrak{b}} 
\newcommand{\me}{\mathfrak{m}}
\newcommand{\ins}{\mathcal{I}}
\newcommand{\slist}{\mathcal{M}}

\newcommand{\cR}{\textcolor{red}}

\newcommand{\cB}{\textcolor{blue}}
\newcommand{\cO}{\textcolor{orange}}

\newcommand{\cP}{\textcolor{purple}}

\definecolor{armygreen}{rgb}{0.01, 0.75, 0.24}
\definecolor{electriccyan}{rgb}{0.0, 1.0, 1.0}
\newcommand{\cGG}{\textcolor{armygreen}}
\newcommand{\cYY}{\textcolor{electriccyan}}

\DeclareMathOperator{\GL}{GL}
\DeclareMathOperator{\SL}{SL}
\DeclareMathOperator{\diag}{diag}
\DeclareMathOperator{\Eval}{Eval}
\DeclareMathOperator{\fail}{fail}
\DeclareMathOperator{\Fix}{Fix}

\newcommand{\tr}{\operatorname{Tr}}
\newcommand{\ps}{\mathfrak{s}}
\newcommand{\MB}{\mathcal{B}}
\newcommand{\PME}{z_1}
\newcommand{\PMZ}{z_2}
\newcommand{\GUE}{t}
\newcommand{\prel}{\omega}
\newcommand{\rvline}{\hspace*{-\arraycolsep}\vline\hspace*{-\arraycolsep}}
\newcommand{\GUFE}{L}
\newcommand{\MZero}{\textcolor{lightgray}{0}}

\newcommand{\stingb}{W}
\newcommand{\stingt}{E}
\newcommand{\asgn}{\, \textcolor{blue(ncs)}{\gets} \,}
\newcommand{\GoingDown}{\FuncSty{GoingDown}}
\newcommand{\GoingUp}{\FuncSty{GoingUp}}
\newcommand{\StandardGenerators}{\FuncSty{StandardGenerators}}
\newcommand{\BaseCase}{\FuncSty{BaseCase}}

\definecolor{inputcolor}{rgb}{0.0, 0.56, 0.0}
\newcommand{\MyIn}[1]{{\color{inputcolor}{#1}}}
\newcommand{\InputT}{\color{inputcolor}$\blacktriangleright$}
\definecolor{outputcolor}{rgb}{1.0, 0.62, 0.0}
\newcommand{\MyOut}[1]{{\color{outputcolor}{#1}}}
\newcommand{\OutputT}{\color{outputcolor}$\blacktriangleright$}
\definecolor{failcolor}{rgb}{0.89, 0.15, 0.21}
\newcommand{\MyFail}[1]{{\color{failcolor}{#1}}}

\newcommand{\MyInT}[1]{{\color{inputcolor}{#1}}}
\newcommand{\MyOutT}[1]{{\color{outputcolor}{#1}}}
\newcommand{\MyOutTT}[1]{{\color{failcolor}{#1}}}
\newcommand{\InLineCode}[1]{{\texttt{#1}}}   
\newcommand{\BC}{\mathcal{L}}
\newcommand{\SLFi}{\text{(SL1)}}
\newcommand{\SLSe}{\text{(SL2)}}
\newcommand{\SLTh}{\text{(SL3)}}
\newcommand{\SLFo}{\text{(SL4)}}
\newcommand{\SLFiv}{\text{(SL5)}}
\newcommand{\SLSi}{\text{(SL6)}}
\newcommand{\SLSev}{\text{(SL7)}}
\newcommand{\step}{\text{phase}}

\newcommand{\steps}{\text{phases}}

\newcommand{\ppd}{\mathrm{ppd}}
\newcommand{\MN}{N}
\newcommand{\GUI}[1]{(#1)}

\newcommand{\cE}{a}  
\newcommand{\cM}{g}   
\newcommand{\cH}{c}  
\newcommand{\cT}{\tilde{g}}  

\newcommand{\Pkg}[1]{\textsf{#1}}
\newcommand{\GAP}{\Pkg{GAP}}
\newcommand{\Magma}{\Pkg{Magma}}

\newcommand\Defn[1]{\emph{\color{blue}#1}}

\newcommand\algoref[1]{\hyperref[#1]{\FuncSty{#1}}}

\title{Constructive Recognition of Special Linear Groups}
\author{
  Max Horn\footnote{RPTU Kaiserslautern-Landau, Germany.
    email: mhorn@rptu.de}
  \and
  Alice C. Niemeyer\footnote{RWTH Aachen University, Germany.
    email: alice.niemeyer@mathb.rwth-aachen.de}
  \and
  Cheryl E. Praeger\footnote{University of Western Australia, Australia.
    email: cheryl.praeger@uwa.edu.au}
  \and
  Daniel Rademacher\footnote{RWTH Aachen University, Germany.
    email: rademacher@art.rwth-aachen.de}
}

\date{April 2023}

\begin{document}

\maketitle

\begin{abstract}
We introduce a new constructive recognition  algorithm
for finite special linear groups in their natural representation. Given
a group $G$ generated by a set of $d\times d$ matrices
over a finite field $\F_q$,  known to be isomorphic
to the special linear group $\SL(d,q)$, the algorithm
computes a special generating set $S$ for $G$. These generators
enable efficient computations with the input group, including 
solving the word problem. Implemented in the computer algebra system 
\GAP, our algorithm outperforms existing state-of-the-art 
algorithms by a significant margin. A detailed complexity analysis 
of the algorithm will be presented in an upcoming publication.
\end{abstract}

\section{Introduction}

During the  inaugural international conference  ``Computational Group
Theory''  held  in  Ober\-wolf\-ach   in  1988,  Neub\"user  highlighted  the
necessity for  efficient algorithms for working with matrix groups
over finite fields.  He challenged the  audience to design an 
algorithm to determine whether a given finite
set  of $d\times d$  matrices over  a finite
field with $q$ elements generates a group containing the special linear
group $\SL(d,q)$.

Neub\"user's initial  query was answered  in 1992 by Peter  Neumann and
the third  author \cite{NeuPrae}.  During  the three decades  since,
a research effort called \emph{the Matrix
Group  Recognition Project} to design algorithms for matrix group computation
has  seen  substantial progress, \cite{AMGR}.  Numerous
algorithms  have been  developed and  implemented in  computer algebra
systems  such as  \GAP~\cite{GAP} and  \Magma~\cite{Magma}, enabling  efficient manipulation  of
matrix  groups of  dimensions as  large as  several thousand.

The  development   of  algorithms   for  matrix  groups   has followed
two main approaches, often referred to as the soluble radical approach and the
Composition Tree approach, the latter being the focus here. Both approaches
have yielded
significant  breakthroughs,  notably  Babai, Beals and Seress  \cite{BBS}
show that the soluble radical approach computes a data
structure which allows to compute with
a matrix group given abstractly as a black box group in 
 randomized polynomial time, using oracles for factoring and discrete log.
 The Composition Tree method is described in \cite{BHGO,GAPCompTree}
and computes a data structure known  as a \emph{Composition Tree} for a given
finite  matrix  group  $G$.  This   composition  tree  serves  as  the
fundamental  tool for  further computations  with $G$,  and exhibits  a
composition series of $G$. The complexity of this method
has been shown to be 
polynomial the degree of the matrix group
by Holt et al \cite{MainCompositionTree} if the  composition tree  avoids
certain  finite  simple  groups.

Given this remarkable progress the reader might wonder why 
we revisit a topic
for which efficient algorithms already exist.  Having efficient
algorithms available is particularly crucial for groups closely
related to the composition factors. As the classical groups yield
infinite families of finite simple groups, improvements in algorithms
for these groups  cascade throughout the entire composition tree, amplifying their impact 
across the structure.  Algorithms to compute a composition tree face two challenges
when confronted with a matrix group: firstly, resolving the \emph{naming
problem}, which involves determining the finite simple group to which
the given group is isomorphic, and secondly, tackling the \emph{constructive
recognition problem}. The constructive recognition problem takes as
input a subgroup $G$ of $\GL(d,q)$ known to be isomorphic to a
particular almost simple group and aims to find an isomorphism to a
standard copy of the almost simple group.  Solutions to constructive
recognition problems facilitate computations with the input group.

The  constructive recognition problem we treat in this paper
takes as input a subgroup $G$ of $\GL(d,q)$ known to contain $\SL(d,q)$
and aims to find an isomorphism to a  standard copy of $\SL(d,q)$. More
precisely, we aim to compute a particular set of matrices, 
called \emph{standard generators}
in the standard copy of $\SL(d,q)$ as words in the generators of $G$.

Several algorithms have  been developed to address  the naming problem
for  classical groups,  including \cite{NeuPrae,NCR}.  The
initial  algorithms  devised  to solve  the  constructive  recognition
problem  for   classical  groups  date  back   to  1995 \cite{CRCELLER}.
Subsequently,  numerous algorithms  have  been proposed  \cite{CRODD, SL2, 
BlackBoxInvolutions, KS},
with the current state-of-the-art algorithm being the one developed by
Dietrich,  Leedham-Green, L\"ubeck  and  O'Brien \cite{CREVEN}.   This
algorithm descends recursively along a  chain of subgroups, seeking at
each  step to  construct  a  classical subgroup  of  roughly half  the
dimension.

Over a decade  ago, our colleagues \'Akos Seress  and Max Neunh\"offer
proposed a novel approach to  the constructive recognition problem for
classical  groups,  based   on  the  concept  of   elements  that  act
irreducibly on a subspace of dimension roughly $\log(d)$ and fix
pointwise a complementary subspace.
We call such elements \emph{stingray elements}.
\'Akos Seress asked
the second and third authors to estimate the proportion of such elements in finite
classical groups, see \cite{SR1}. Moreover, he worked with Yal\c{c}\i{}nkaya and the third author
on estimating the probability that a stingray element and its conjugate would generate a classical group
of smaller dimension \cite{PSY}.

Despite their  promising  ideas,  further development  was
halted  following Seress's  passing  and  Neunh\"offer's departure  from
academia. We believed their ideas to be very promising and in a series
of papers \cite{SR1,SR2,SR3,C1,C9,GEN} have paved the way for a complexity analysis of
the algorithms presented here.

We  present in  this paper  a new  constructive recognition  algorithm
based on the theory of stingray elements
tailored  for the  special linear  groups in degree $d\geq 4$. Our algorithm
demonstrates an unexpected  speed and practical
efficiency, outperforming the current  state-of-the-art algorithm by a
significant margin. In particular, we prove the following theorem.

\begin{theorem}\label{MainResultsSummary}
Let $\langle X \rangle = G = \SL(d,q)$ and $\epsilon \in (0,1)$.
\algoref{StandardGenerators} is a one-sided Monte Carlo algorithm which
given input $G$ and $\epsilon$ outputs with probability at least $1-
\epsilon$ an MSLP $\slp$ and base change matrix $\BC$ such that $\slp$
evaluates from $X^\BC$ to the standard generators of $G^\BC$.
\end{theorem}

A detailed complexity analysis of our new algorithm is almost complete.
We have also successfully adapted our method to the other families of
classical groups. All of this will be presented in follow-up publications.
In particular, we hope to establish the following result.
\begin{conjecture}
The complexity of \algoref{StandardGenerators} as stated in
\cref{MainResultsSummary} is 
\[
\mathcal{O}(d^3 \log(d) + d^2 \log(d) \log(\log(d)) \log(q) + \frac{\log(d)}{\log(\log(d))} \xi + \zeta(q))
\]
 where $\xi$ denotes an upper
 bound on the number 
of field operations for computing a random element in $\SL(d,q)$ and $\zeta(q)$ denotes an 
upper bound on the number of field operations for constructively recognising $\SL(4,q)$.
\end{conjecture}

\section*{Acknowledgements}

The first, second and fourth authors acknowledge support by the German Research Foundation (DFG) -- Project-ID 286237555 -- within the SFB-TRR 195 ``Symbolic Tools in Mathematics
and their Applications''.
The second and third author acknowledge the support of
Australian Research Council Discovery program grant DP190100450 
and the third author acknowledges the support of
Australian Research Council Discovery program grant DP230101268.

Note that the content of this publication is a condensed version
of some of the chapters in the fourth author's Doctoral dissertation.

The fourth author thanks Eamonn O'Brien for his hospitality during a research stay at the University of Auckland, and many relevant discussions.

\section{Preliminaries}

Throughout this paper, we use the following notations:
\begin{itemize}
\item $q = p^f$ denotes a prime power.
\item $\F_p$ and $\F_q$ are finite fields of order $p$ resp.\ $q$.
\item $(\prel_1, \dotsc, \prel_f)$ is an $\F_p$-basis of $\F_q$.
\item $n$ and $d$ are natural numbers satisfying $n\leq d$.
\item By $E_{i,j}(\lambda)$ we denote an elementary matrix with entry $(i,j)$ equal to $\lambda$.
\end{itemize}

\begin{definition}
Let $V$ be a $d$-dimensional vector space with basis $\mathcal{B} = \{
v_1, \dotsc , v_d \}$. 
\begin{itemize}
\item 
$\langle \cdot \mid \cdot \rangle$ denotes the
\Defn{standard scalar product} on $V$ with respect to $\mathcal{B}$,
i.e., 
$\langle \sum \lambda_i v_i, \sum \mu_i v_i \rangle = \sum \lambda_i \mu_i$.

\item \label{transvection}
For $v,w \in V$ linearly independent, $T_{v,w}$ denotes the \Defn{transvection} $x \mapsto
x + \langle x \mid v \rangle w$. 
\end{itemize}
\end{definition}
Note that the elementary matrix
$E_{i,j}(\lambda)$ is equal to the matrix of the transvection $T_{\lambda v_i,
v_j}$ with respect to $\mathcal{B}$.

\begin{definition}\label{IteratedLogarithm}
Let $k$ be a non-negative integer. The \Defn{iterated logarithm} $\log*(k)$ is defined recursively as follows
\[
\log^*(k) \coloneqq
\begin{cases}
0, & \text{ if } k \leq 1 \text{ and} \\
1 + \log^*( \log(k) ), & \text{ if } k> 1.
\end{cases}
\]
\end{definition}

\begin{definition}\label{StandardGeneratorsSL}
Let $S \subset \SL(d,q)$. Then $S$ is a set of \Defn{standard
generators} for $\SL(d, q)$ if $S$ is conjugate to the following set
consisting of $2f + 2$ elements:
\begin{itemize}
\item $E_{1,2}(\prel_i)$ for $1 \leq i \leq f$,
\item $E_{2,1}(\prel_i)$ for $1 \leq i \leq f$,
\item a permutation matrix $z_1$ corresponding to the permutation $(1, d, d - 1, \dotsc, 2)$ with the entry $(z_1)_{1, d}$ changed to $-1$ if $d$ is even and
\item a permutation matrix $z_2$ corresponding to the permutation $(2, d, d - 1, \dotsc, 3)$ with the entry $(z_2)_{2, d}$ changed to $-1$ if $d$ is odd.
\end{itemize}
\end{definition}

The $(-1)$-entries are introduced to ensure that the elements are contained in $\SL(n, q)$.

A notation we use frequently throughout is the following. 

\begin{definition}\label{StandardEmbedded}
Let $G \leq \GL(d,q)$ and $U \leq G$. Then $U$ is \Defn{stingray embedded
of degree $n$ in $G$} for $n \leq d$ if there exists a group
$H \leq \GL(n,q)$, a matrix $\BC \in \GL(d,q)$ and an
isomorphism $\sigma_\BC$ from $H$ into $U^\BC$ such that
\[
\sigma_{\BC} \colon H \to U^\BC,\ a \mapsto \diag(a, I_{d-n}) =
\begin{pmatrix}
a & 0 \\
0 &  I_{d-n} \\
\end{pmatrix}.
\]
We denote a stingray embedding by
\[
U^\BC =
\begin{pmatrix}
H & 0 \\
0 &  I_{d-n} \\
\end{pmatrix} \leq G^\BC
\]
where $G^\BC$ is the conjugation by $\BC$, i.e.\ $G^\BC = \{ \BC^{-1} g \BC \mid g \in G \}$.
\end{definition}

\begin{remark} \label{Naming}
A naming algorithm for a family of group (e.g. the family of all classical
groups) takes as input a subset $X\subset \GL(d,q)$ and decides whether the
group $G=\langle X \rangle$ generated by $X$ is isomorphic to a member of that
family. This is also sometimes referred to as ``non-constructive
recognition'', as the group $G$ is ``recognized'' to have a specific
isomorphism type, but without actually constructing an isomorphism.

Some of the algorithms presented in this paper 
need to decide if $G$ is equal (and not just isomorphic) to $\SL(d,q)$. This
can certainly be achieved using a naming algorithm for classical groups such
as the ones presented in \cite{NeuPrae,NCR}.
\end{remark}

Finally, we remark on some conventions we employ to enhance the readability
of the algorithms we present.

\begin{remark}\label{RandomisedAlgos}
  The algorithms we present are randomised algorithms. They rely on
  a procedure to select elements from a group $G$ which are nearly
  uniformly distributed and independent. Such an algorithm
  has been introduced by \cite{RandomisedAlgorithms}. A practical version
  is presented in \cite{RandomElements}.
\end{remark}

\begin{remark}\label{GlobalCounter}
Some randomised algorithms described in this publication call other
randomised algorithms. In order to bound the overall number of computed
random elements, we use a ``global'' variable $\MN$. That means if a
randomised algorithm $A$ uses the variable $\MN$ and calls a randomised
algorithm $B$ with input $\MN$, then $A$ and $B$ ``share'' the identical
$\MN$, i.e.\ if $\MN$ is decreased by $B$ and $B$ returns $\MN$, then
the $\MN$ of $A$ is ``updated'' and thus also decreased. The variable
$\MN$ is very useful in this context to bound the maximal number of
random elements computed by the input $\MN$.
\end{remark}

\begin{remark}\label{OutputCheck}
Some algorithms of this publication call other randomised algorithms
which can either return a valid output or $\InLineCode{fail}$. Whether
an algorithm returns valid output or $\InLineCode{fail}$ needs to be
checked before the output is given as an argument to the next function.
If a randomised algorithm $A$ calls a randomised algorithm $B$ and $B$
returns $\InLineCode{fail}$, then $A$ also terminates and returns
$\InLineCode{fail}$. The check whether $B$ returns $\InLineCode{fail}$
requires an if statement in an implementation. To provide better
readability of algorithms we sometimes omit this if-statement and refer
to this remark instead.
\end{remark}

Finally, we  present proofs for  the correctness of the  algorithms we
present.  However, the  complexity analysis of our  algorithms will be
published in a forthcoming paper. The  reason for this is that many of
the  algorithms  we present  rely  on  various  other results  in  the
literature and  drawing all  of these together  to present  a complete
complexity analysis would increase the length of this paper too much.
\section{MSLP}

Many of the algorithms in the  literature store elements of a group as
words in a given generating set in an efficient way, called a straight
line program, namely as a word  in certain subwords.  In this section,
we describe a modification to straight line programs, which also records the
memory usage of a straight line program.

\begin{definition} \label{MSLP}
  Let $G=\langle X \rangle$ be a group,
  $\ssize \in \mathbb{N}_0$ and let $\slist =
[\me_1,\dotsc,\me_\ssize]$ be an ordered list of $\ssize$ elements of
$X$ such that $\langle \me_1,\dotsc,\me_\ssize\rangle = G$. A
   \Defn{straight-line program with memory (MSLP)} is a sequence
$\slp = [\ins_1,\dotsc,\ins_\slength]$ of \Defn{instructions}, where for
$1
\leq r \leq \slength$ and $i,j,k \in \{1,\dotsc,\ssize\}$
\begin{enumerate}[label={(\roman*)}]
\item $\me_k \leftarrow \me_i$. This instruction stores $\me_i$ in the list $\slist$ in slot $k$.
\item $\me_k \leftarrow \me_i \cdot \me_j$. This instruction stores $\me_i \cdot \me_j$ in the list $\slist$ in slot $k$.
\item $\me_k \leftarrow \me_i^{-1}$. This instruction stores $\me_i^{-1}$ in the list $\slist$ in slot $k$.
\item Show($A$) where $A \subseteq \{1,\dotsc,\ssize \}$. The slots specified by $A$ are displayed.
\end{enumerate}

The number $\ssize \in \mathbb{N}_0$ is the \Defn{memory quota} of
$\slp$ and $\slp$ is a $\ssize$-MSLP. The number $\slength \in
\mathbb{N}_0$ is the length of $\slp$. The empty sequence is permitted
with length 0.
\end{definition}

The \Defn{evaluation of an MSLP}
$\slp$ amounts to executing its instructions recursively. 
Suppose $G$ is a group and $X$ is a sequence of elements of $G$. If 
the memory $\slist$ is initialised with the sequence $X$ when
$\slp$ is evaluated, then
we say that $\slp$ is \Defn{evaluated in $X$}.
The instruction
Show($A$) can be employed to output multiple elements of $G$ stored in
one of the $\ssize$ memory slots specified by $A$. For more details,
see \cite[Section~2]{MSLP2024}.

In the evaluation process, $\slist$ is used as memory for the elements
that are needed. The instruction $(i)$ of \cref{MSLP} can be used
to overwrite a slot and minimize the memory quota. The idea of an MSLP
is to reduce the memory required for evaluating a particular word in a
group by repeatedly overwriting memory slots which are no longer used.
The length $\slength$ of an MSLP describes the number of operations
during the entire evaluation.

\begin{remark}
The instructions of \cref{MSLP} use elements of
$\{1,\dotsc,\ssize \}$. The memory $\slist$ is secondary in the
description. This implies that an MSLP is independent of the group,
i.e.\ $\Eval_{\slist}(\slp)$ can be computed for every $\slist \in
G^\ssize$ and every group $G$. Hence, it is possible to encode an
element as a word in one group and to evaluate the constructed SLP in
another group.
\end{remark}

\section{Outline of the algorithm} \label{OutlineOfTheAlgorithm}
The goal of this section is to present an overview of the algorithm to
recognise a special linear group constructively. The overarching algorithm is
Algorithm \StandardGenerators.
This algorithm in turn relies on several subprocedures whose purpose is
discussed here.

The input of Algorithm \StandardGenerators\ is a set $X$
with $X \subseteq \GL(d,q)$
 with $d > 4$ such that $G :=
\langle X \rangle = \SL(d,q)$ is a special linear group in its natural
representation.
The output is a
matrix $\BC \in \GL(d,q)$ and an MSLP $\slp$ such that
when $\slp$ is evaluated in $X^\BC$,
the output of $\slp$ is a particular generating set of $\SL(d,q)^\BC$,
called the \Defn{standard generators} of $\SL(d,q)^\BC$.
The definition of standard generators is given in \cref{StandardGeneratorsSL}.

Algorithm \StandardGenerators\ calls
three basic subalgorithms, namely:
\begin{enumerate}[label={(\arabic*)}]
\item Algorithm \GoingDown\ is a one-sided Monte Carlo algorithm to construct a stingray embedded subgroup $\SL(4,q)$ of $G$.
\item Algorithm \BaseCase\ is a one-sided Monte Carlo algorithm to recognise $\SL(4,q)$ constructively.
\item Algorithm \GoingUp\ is a one-sided Monte Carlo algorithm to construct standard generators of $G$ using the constructed standard generators of a stingray embedded subgroup $\SL(4,q)$.
\end{enumerate}

\subsection*{\GoingDown\ algorithm}
The input of Algorithm \GoingDown\ is the set $X$ described above.
The output
 is a matrix $\BC\in \GL(d,q)$, an MSLP $\slp$ and a subgroup
$\langle X_U \rangle = U \leq G$ with $U \cong \SL(4,q)$ which is stingray embedded, that is
\[
U^\BC =
\begin{pmatrix}
    \SL(4,q) & 0 \\
    0 & I_{d-4}
\end{pmatrix} .
\]
Moreover, evaluating the MSLP $\slp$ in $X$ yields the generators $X_U$ of $U$.

Algorithm \GoingDown\ is a randomised algorithm and
relies on finding so-called \Defn{stingray elements}
(see \cref{StingrayElement}) by random selection of
elements of $G$. Therefore there is an additional input $\MN \in \N$
which is used as an upper bound for the maximal number of
random selections. If the limit $N$ is exceeded, then Algorithm \GoingDown\
terminates returning $\InLineCode{fail}$. Moreover,
Algorithm \GoingDown\ is a one-sided Monte-Carlo algorithm, therefore we
also discuss how to choose $\MN \in \N$ for a given $0 < \epsilon < 1$
such that the \GoingDown\ algorithm succeeds with probability $1 -
\epsilon$.

Algorithm \GoingDown\ starts by setting $U_1:=G$, $d_1:=d$, $\BC_1:=I_d$.
then repeatedly calls the subroutine \algoref{GoingDownBasicStep} via
\[
U_{i+1}, d_{i+1}, \BC_{i+1}, \slp_{i+1}, \MN' := \algoref{GoingDownBasicStep}(U_i, d_i, \BC_i, \MN),\]
where for $i>1$ we assume that $U_i$ is stringray embedded of degree $d_i$ in $U_{i-1}$. The output satisfies that
$U_{i+1}$ is stingray embedded of degree
$d_{i+1}$ (and hence in $G$). Moreover, 
$d_{i+1} \leq 4 \lceil \log(d_i) \rceil$. Next, $\BC_{i+1} \in \GL(d,q)$ is a
base change matrix, and $\slp_{i+1}$ is an MSLP such that
when $\slp_{i+1}$ is evaluated in the generators of $U_i$
it returns generators of $U_{i+1}$. Finally, $\MN'$ is equal to
$\MN$ minus the number of random selection performed by the subroutine, and is used as $\MN$ for the next step.
In summary, we have
\[
U_{i+1}^{\BC_{i+1}} = \begin{pmatrix}
    \SL(d_{i+1},q) & 0 \\
    0 & I_{d-d_{i+1}}
\end{pmatrix} \leq G.
\]
The subroutine \algoref{GoingDownBasicStep} 
is repeated until a group isomorphic to $\SL(4,q)$ is reached.
At this point we set $k := i+1$ and return the descending chain
\[
\SL(4,q)\cong U_k \leq U_{k-1} \leq \dotsc \leq U_2 \leq U_1 = G = \SL(d,q).
\]
Since $d_{i+1} \leq 4 \lceil \log(d_i) \rceil$ holds,
$k$ is roughly $\log^*(d)$, where $log^*$ denotes
the iterated logarithm, see \cref{IteratedLogarithm}.
The descending chain thus computed is a
\Defn{descending recognition chain} of $G$:
\begin{definition}\label{DescendingRecognitionChain}
Let $\SL(d,q)$ be a special linear group in its natural representation.
Then a \Defn{descending recognition chain} of $\SL(d,q)$
is a descending chain of subgroups
\[ U_k \leq U_{k-1} \leq \dotsc \leq U_2 \leq U_1 = \SL(d,q), \]
where $U_i \cong \SL(d_i,q)$ is stingray embedded $U_{i-1}$ (and thus in $\SL(d,q)$) for $i>1$, we have $d_{i+1}
\leq 4 \lceil \log(d_{i}) \rceil$ for all $1 \leq i < k$, and $d_k=4$. 
\end{definition} 

The basic step uses stingray elements (defined in
\cref{sec:StingrayElements}) as follows: In $U_i$ we seek two
stingray elements $s_1, s_2 \in U_i$, using
\algoref{FindStingrayElement} described in \cref{sec:StingrayElements}.
The elements are chosen such that $\langle s_1, s_2 \rangle \cong
\SL(d_{i+1},q)$ holds with high probability. Then
$\langle s_1, s_2 \rangle$ is automatically
stingray embedded $U_i$ (and thus in $\SL(d,q)$),
and hence we can define  $U_{i+1} := \langle s_1, s_2\rangle$.
The base change matrix to achieve this block structure is also computed
by the algorithms of \cref{sec:StingrayElements}. 

Since the basic
step relies on finding stingray elements by randomised procedure,
the \GoingDown\ algorithm is a randomised algorithm. In fact, the
\GoingDown\ algorithm returns $\InLineCode{fail}$ if it does not succeed after $\MN$
random element selections of elements from $U_i$.
The probability that the \GoingDown\
algorithm succeeds relies completely on $\MN$. We analyse this
dependency in a forthcoming paper.

\begin{remark}
  It is not possible to reach $\SL(2,q)$ using
  \algoref{GoingDownBasicStep} 
because the degrees $d_i$ of $U_{i} \cong \SL(d_i,q)$ are too
small to find stingray elements generating special linear
groups. Therefore, we terminate   \algoref{GoingDownBasicStep} 
when we
encounter $\SL(4,q)$. See \cref{StingrayNotToDimension2} for details.
\end{remark}

\begin{remark}
Note that $\SL(3,q)$ in its natural representation is also not handled
by our algorithm, but for this case one can e.g. use \cite{CRSL3}.
\end{remark}

\subsection*{\BaseCase\ algorithm}

The \BaseCase\ algorithm is the second subalgorithm of the
\StandardGenerators\ algorithm. Let $G := \SL(d,q)$ be the input of
\StandardGenerators. Then the input of the \BaseCase\ algorithm is
a generating set $X$ of a group $\SL(4,q) \leq G$. Moreover, a
base change matrix $\BC
\in \GL(d,q)$ is known such that $\langle X \rangle^\BC$ is
stingray embedded with degree $4$ in $G^\BC$. Since the \BaseCase\
algorithm is also randomised we again use $\MN$ as additional input to
restrict the maximal number of random element selections to $\MN$. The
output of the \BaseCase\ algorithm is an MSLP $\slp$ and a base
change matrix $\BC' \in \GL(d,q)$ such that if $\slp$ is evaluated in
$X^{\BC'}$, the output are the standard generators of
$\SL(2,q)^{\BC'}$ as in \cref{StandardGeneratorsSL}.

To do this,  the \BaseCase\ algorithm uses a simplified version of the
\GoingDown\ algorithm of \cite{CREVEN,CRODD} to find a suitable $\SL(2,q)$
subgroup. It then calls an efficient constructive recognition 
algorithm for $\SL(2,q)$,
e.g.\ the algorithm described in \cite{CRSL2}.

Properties of the \BaseCase\ algorithms e.g.\ whether it is 
randomised or  requires a
discrete logarithm oracle depend on the choice of
the constructive recognition algorithm for the base case group. Note
that a constructive recognition algorithm employed for base case groups
can readily be exchanged by another should a more efficient 
algorithm 
 become available.

The state of the art algorithm for constructive recognition of
$\SL(2,q)$ is given in \cite{CRSL2} which is a randomised algorithm and
uses the discrete logarithm oracle. So far, it is unknown whether it is
possible to recognise $\SL(2,q)$ constructively without the discrete
logarithm oracle.

As for the \GoingDown\ algorithm the success of the \BaseCase\
algorithm depends on the choice of $\MN$ and for a given $\epsilon \in
(0,1)$ we can select $\MN$ such that the \BaseCase\ algorithm succeeds
with probability at least $1-\epsilon$. The choice of $\MN$ of the
\BaseCase\ algorithm is not discussed in this publication, instead we
refer to the literature, see \cite{CRSL2}.

\subsection*{\GoingUp\ algorithm}
The last subalgorithm of \StandardGenerators\ is the \GoingUp\
algorithm. The input of this algorithm is $X$ with $\langle X \rangle =
G = \SL(d,q)$ in its natural representation and the standard generating
set $S$ of a stingray embedded subgroup $\langle S \rangle = H
\leq G^\BC$, where $H$ is a base case group 
for
a known base change matrix $\BC \in \GL(d,q)$. Moreover, 
MSLP are known to express the standard
generators $S$ of $H$ 
as words in $X^\BC$, i.e.\ a
constructive recognition algorithm has been used on $H$. The \GoingUp\
algorithm is also randomised and, therefore, we use $\MN$ as input to
allow at most $\MN$ random element selections. The outputs of the
\GoingUp\ algorithm are an MSLP $\slp$ and a base change matrix $\BC'$
such that, if $\slp$ is evaluated in $(S \cup X)^{\BC \BC'}$, the
output are the standard generators of $G^{\BC \BC'} = \SL(d,q)$.

Similar to the \GoingDown\ algorithm, the \GoingUp\ algorithm uses a
basic step repeatedly until the standard generators of $G$ have been
constructed. The input of the \GoingUp\ basic step are $X$ and
$\widetilde{S}$ with $\langle \widetilde{S} \rangle = \widetilde{H} \leq
G^{\widetilde{\BC}}$ where $\widetilde{H} \cong \SL(n,q)$ is stingray
embedded 
into $\GL(d,q)$ 
for a known base change matrix $\widetilde{\BC} \in \GL(d,q)$.
The standard generators $\widetilde{S}$ of $\widetilde{H}$ can be
written as words in $X^{\widetilde{\BC}}$ and are encoded in an MSLP
$\widetilde{\slp}$. Moreover, the base change matrix $\widetilde{\BC}$
as well as $\MN$ to control the maximal number of random element
selections are additional inputs of the \GoingUp\ basic step. The output
of the \GoingUp\ basic step is a base change matrix $\BC \in \GL(d,q)$
and MSLP which evaluates in $\widetilde{S} \cup X$ to the standard
generators of another stingray embedded special linear subgroup $K$ of
$G^\BC$. Since the output of the \GoingUp\ basic step is an MSLP
evaluating to the standard generators of a subgroup of $G$ we denote
this by saying that the input of the $i$-th basic step is given by $X$,
$S_i, \BC_i$ and $\MN$ and the output group is given by $S_{i+1}$ and
$\BC_{i+1}$ with $\langle S_{i+1} \rangle= H_{\GUI{i+1}}$ resulting in
\[
S_{i+1},\BC_{i+1} := \algoref{GoingUpStep}(X,S_i,\BC_i, \MN).
\]
Note that this is not completely accurate as the output is not directly
$S_{i+1}$ but rather is only an MSLP which evaluates to the standard
generators $S_{i+1}$ of $H_{\GUI{i+1}}$. We ignore this small inaccuracy
since we know that the MSLP evaluates to $S_{i+1}$. So far we only noted
that $H_{\GUI{i+1}} \leq G^{\BC_{i+1}}$ holds but we know more about the
output of the \GoingUp\ basic step which is $H_{\GUI{i}}^{\BC_i^{-1}}
\leq H_{\GUI{i+1}}^{\BC_{i+1}^{-1}} \leq G$ and $H_{\GUI{i+1}} \cong
\SL(n_{i+1},q)$. By setting $H_{\GUI{0}} = H$ this yields a chain of
subgroups called an \Defn{ascending recognition chain}
\[
H = H_{\GUI{0}}^{\BC_0^{-1}}
\leq H_{\GUI{1}}^{\BC_1^{-1}}
\leq \dotsc
\leq H_{\GUI{\ell-1}}^{\BC_{\ell-1}^{-1}}
\leq H_{\GUI{\ell}}^{\BC_{\ell}^{-1}}
= G^{\BC_{\ell}^{-1}}.
\]
In most cases the length of an ascending recognition chain is larger
than the length of the descending recognition chain. Since the standard
generators of each $H_{\GUI{i}}$ are known and $H_{\GUI{\ell}} = G$, the
standard generators of $G$ are constructed by the \GoingUp\ algorithm.

The \GoingUp\ basic step is more complicated than the \GoingDown\
basic step and consists of seven $\steps$. Therefore, we are not
discussing the details of the \GoingUp\ basic step in this section and
refer to a detailed description in \cref{GoingUpSL}. Instead we continue
this section by displaying the input and output of the \GoingUp\ basic
step in matrix form.

Each of the $H_{\GUI{i}}$ is a stingray embedded special linear group, i.e.\
\[
H_{\GUI{i}} = \begin{pmatrix}
    \SL(n_i,q) & 0 \\
    0 & I_{d-n}
\end{pmatrix} \leq G^{\BC_i}
\]
and the standard generators of $H_{\GUI{i}}$ are known. The \GoingUp\
basic step then proceeds roughly as follows. We identify an element $c
\in G^{\BC_i}$ such that
\[
\langle H, H^c \rangle := \begin{pmatrix}
    \SL(\mathcal{N}(d,n),q) & 0 \\
    0 & I_{d-\mathcal{N}(d,n)}
\end{pmatrix} \leq G^{\BC_{i+1}}
\]
for a base change matrix $\BC_{i+1} \in \GL(d,q)$ such that $\langle H,
H^c \rangle$ is stingray embedded as displayed above. The function
$\mathcal{N}(d,n)$ is given by $\min\{ 2n-1, d \}$.

Clearly, $H \leq \langle H, H^c \rangle$ and the standard generators for
$H$ are known. Using this knowledge and $c$ we can very carefully
construct standard generators for $\langle H, H^c \rangle \cong
\SL(\mathcal{N}(d,n),q)$. Then we set $H_{\GUI{i+1}} := \langle H, H^c
\rangle$ and the standard generators for $H_{\GUI{i+1}}$ can be derived.
Since we want to avoid as many matrix operations as possible the
standard generators of $H_{\GUI{i+1}}$ are only encoded in an MSLP
instead of performing the actual matrix computations. As for the
\GoingDown\ algorithm we introduce a notion for a ascending recognition
chain in the next definition.

\begin{definition}\label{AscendingRecognitionChain}
Let $\langle X \rangle = G = \SL(d,q)$. Then a ascending chain of subgroups of $G$ with
\[
H = H_{\GUI{0}} \leq H_{\GUI{1}} \dotsc \leq H_{\GUI{\ell-1}} \leq H_{\GUI{\ell}} = G,
\]
where $H_{\GUI{i+1}} \cong \SL(n_{i+1},q)$ of the same type as $G$ is
stingray embedded in $G$ and $H$ is a base case group, is a
\Defn{ascending recognition chain}.
\end{definition}

\subsection*{\StandardGenerators\ algorithm}
In this section we describe the fundamental structure of the
constructive recognition algorithm of this publication in a single
algorithm \StandardGenerators. The \StandardGenerators\ algorithm
uses the three subalgorithms \GoingDown, \BaseCase\ and \GoingUp\
outlined in the previous sections. The input and output of the
\StandardGenerators\ algorithm are as described in the introduction of
\cref{OutlineOfTheAlgorithm}.

The \StandardGenerators\ algorithm deals as follows with the input
$\langle X \rangle = G \coloneqq \SL(d,q)$:
\begin{enumerate}[label={(\arabic*)}]
    \item If $G$ is a base case group, then call the \BaseCase\ algorithm on $G$ and return the output.
    \item Call the \GoingDown\ algorithm with input $G$ to construct a base case group $U$ with $U \leq G$.
    \item Call the \BaseCase\ algorithm with input $U$ to recognise $U$ constructively.
    \item Call the \GoingUp\ algorithm with input $U$ to express the standard generators of $G$ as words in $X$ in an MSLP.
\end{enumerate}

The \StandardGenerators\ algorithm can be displayed in pseudo code as follows.

\begin{algorithm}[H]
\caption{StandardGenerators}
\label{StandardGenerators}
\small{
{\nlnonumber
\IOKw{
\begin{tabular}{@{}ll}
        \KwIn{}& {\InputT} {\normalfont $\MyIn{\langle X \rangle = G} = \SL(d,q)$ in its natural representation, with $d\geq 4$} \\
        & {\InputT} {\normalfont $(\MyIn{\MN_1}, \MyIn{\MN_2}, \MyIn{\MN_3}) \in \N$} \\
        \KwOut{}& \MyFail{fail} OR $\MyOut{(\BC, \slp)}$ where \\
        & {\OutputT} {\normalfont $\MyOut{\BC} \in \GL(d,q)$ is a base change matrix and} \\
        & {\OutputT} {\normalfont $\MyOut{\slp}$ is an MSLP such that if $\MyOut{\slp}$ is evaluated on $\MyIn{X}^{\MyOut{\BC}}$, then the} \\
        & {\normalfont standard generators of $\MyIn{G}^{\MyOut{\BC}}$ are computed} \\
\end{tabular}
} \\
}
\nlnonumber
\LinesNumbered
\setcounter{AlgoLine}{0}
\Begin($\FuncSty{StandardGenerators} {(} \MyInT{G, (\MN_1, \MN_2, \MN_3)} {)}$)
{
	$U, \BC_1, \slp_1 := \FuncSty{GoingDown}(G,\MN_1)$ \;
	$\slp_2, \BC_2 := \algoref{BaseCase}(U,\MN_2)$ \;
	$\slp_3, \BC_3 := \algoref{GoingUp}(G^{\BC_1 \BC_2},U^{\BC_2},\MN_3)$ \;
	Combine $\slp_1$, $\slp_2$ and $\slp_3$ into an single MSLP $\slp$ \;
	$\BC := \BC_1 \BC_2 \BC_3$ \;
	\Return $\MyOutT{(\slp,\BC)}$ \;
}
}
\end{algorithm}

The algorithm \StandardGenerators\ yields the structure for the
constructive recognition algorithms of this publication. In the next
three sections we present the details for the subalgorithms
\GoingDown, \BaseCase\ and \GoingUp.

We finish this section by stating the main theorem of this publication.

\begin{theorem}\label{MainResultsSummary}
Let $\langle X \rangle = G = \SL(d,q)$ and $\epsilon \in (0,1)$.
\algoref{StandardGenerators} is a one-sided Monte Carlo algorithm which
given input $G$ and $\epsilon$ outputs with probability at least $1-
\epsilon$ an MSLP $\slp$ and base change matrix $\BC$ such that $\slp$
evaluates from $X^\BC$ to the standard generators of $G^\BC$.
\end{theorem}

We plan to analyse the overall complexity of this algorithm in a follow-up
paper, hopefully establishing the following:
\begin{conjecture}
The complexity of \algoref{StandardGenerators} as stated in
\cref{MainResultsSummary} is 
\[
\mathcal{O}(d^3 \log(d) + d^2 \log(d) \log(\log(d)) \log(q) + \frac{\log(d)}{\log(\log(d))} \xi + \zeta(q))
\]
 where $\xi$ denotes an upper
 bound on the number 
of field operations for computing a random element in $\SL(d,q)$ and $\zeta(q)$ denotes an 
upper bound on the number of field operations for constructively recognising $\SL(4,q)$.
\end{conjecture}

\section{Stingray elements} \label{sec:StingrayElements}
In this section stingray elements are defined and general properties of
stingray elements are summarised. We start with the definition of
stingray elements and ppd stingray elements.

\begin{definition}\label{PrimitivePrimeDivisor}
Let $\mathfrak{a},\mathfrak{b} \in \N$ with $\mathfrak{a},\mathfrak{b} 
> 1$. A prime $r$ which divides $(\mathfrak{a}^\mathfrak{b}  -1)$ is a
\Defn{primitive prime divisor} (or \Defn{ppd} for short) of
$\mathfrak{a}^\mathfrak{b} -1$ if $r \nmid \mathfrak{a}^i -1$ for all $i
\in \N$ with $i < \mathfrak{b}$.

Let $q=p^f$ be a prime power and $d,m \in \N$. An element $g \in \GL(d,
q)$ is a \Defn{$\ppd(d, q; m)$-element} if there is a primitive prime
divisor $r$ of $q^m - 1$ such that $r$ divides the order of $g$.
Define $\Phi(m,q)$ to be the product of all primitive prime divisors
of $q^m-1$, counting their multiplicities.
\end{definition}

\begin{definition}\label{StingrayElement}
  An element $s \in \GL(d,q)$
	is an \Defn{$m$-stingray element}
  for $1< m \le d$, if $s$ has a $(d-m)$-dimensional
$1$-eigenspace $\stingt_s=\ker(s-1)$ and $s$ acts irreducibly on a complementary
invariant subspace $\stingb_s=\text{\rm im}(s-1) \leq \F_q^d$ of dimension $m$. The space
$W_s$ is the \Defn{stingray body} and $E_s$ is the \Defn{stingray tail}.
A stingray element $s\in \GL(d,q)$ is a \Defn{ppd $m$-stingray element}
if $s$ a $\ppd(d,q;m)$-element.
\end{definition}

\begin{example}\label{ex:StingrayElement}
	Consider the following matrix in $\GL(10,5)$:
\[
\scriptsize\arraycolsep=2pt
\begin{pmatrix}
4 & 1 & 3 & 3 & 2 & 1 & 2 & 4 & 0 & 4 \\
3 & 3 & 4 & 4 & 3 & 4 & 1 & 0 & 3 & 2\\
2 & 1 & 0 & 4 & 3 & 1 & 1 & 2 & 0 & 3\\
3 & 1 & 3 & 4 & 4 & 0 & 2 & 0 & 1 & 3\\
3 & 4 & 1 & 1 & 0 & 3 & 4 & 4 & 1 & 1 \\
1 & 0 & 4 & 4 & 2 & 2 & 1 & 1 & 4 & 2\\
2 & 4 & 2 & 2 & 0 & 3 & 4 & 2 & 1 & 0\\
0 & 2 & 2 & 2 & 4 & 0 & 3 & 4 & 2 & 0\\
3 & 4 & 1 & 1 & 4 & 3 & 4 & 4 & 2 & 1\\
0 & 2 & 2 & 2 & 2 & 1 & 3 & 2 & 1 & 2
\end{pmatrix}
\xrightarrow{\text{change of basis}}
\begin{pmatrix}
0 & 0 & 0 & 4 & \rvline & \MZero & \MZero & \MZero & \MZero & \MZero & \MZero \\
1 & 0 & 0 & 0 & \rvline & \MZero & \MZero & \MZero & \MZero & \MZero & \MZero \\
0& 1 & 0 & 3 & \rvline & \MZero & \MZero & \MZero & \MZero & \MZero & \MZero \\
0 & 0 & 1 & 4 & \rvline & \MZero & \MZero & \MZero & \MZero & \MZero & \MZero \\
\hline
\MZero & \MZero & \MZero & \MZero & \rvline & 1 & \MZero & \MZero & \MZero & \MZero & \MZero\\
\MZero & \MZero & \MZero & \MZero & \rvline &  \MZero & 1 & \MZero & \MZero & \MZero & \MZero\\
\MZero & \MZero & \MZero & \MZero & \rvline &  \MZero & \MZero & 1 & \MZero & \MZero & \MZero\\
\MZero & \MZero & \MZero & \MZero & \rvline &  \MZero & \MZero & \MZero & 1 & \MZero & \MZero\\
\MZero & \MZero & \MZero & \MZero & \rvline &  \MZero & \MZero & \MZero & \MZero & 1 & \MZero\\
\MZero & \MZero & \MZero & \MZero & \rvline &  \MZero & \MZero & \MZero & \MZero & \MZero & 1
\end{pmatrix}
\]
	This element acts irreducibly on a $4$-dimensional subspace of $\F_5^{10}$ and fixes a
6-dimensional subspace pointwise and is therefore a stingray element.
\end{example}

The algorithms presented here require ppd stingray elements.
As the proportion of ppd stingray elements in $\SL(d,q)$ is small, 
we seek elements, called
\Defn{ppd pre-stingray candidates}, from which 
ppd stingray elements can easily be constructed. 

\begin{definition}\label{PrestingrayElement}
  An element $\ps \in \GL(d,q)$ is an \Defn{$m$-pre-stingray candidate} for
  $m\le d$ if the
characteristic polynomial $\chi_{\ps}(x)$ has an irreducible factor
$P(x) \in \F_q[x]$ of degree $m$ over $\F_q$ and no other irreducible
factors of degree divisible by $m$. The irreducible factor $P(x)$ is a
\Defn{stingray factor}. Moreover, if 
$\ps$ is a ppd$(d, q; m)$-element, we say that
$\ps$ is a  \Defn{ppd $m$-stingray element}.
\end{definition}

Note that the degree of a stingray factor of a pre-stingray candidate is
equal to the dimension of the stingray body of a corresponding stingray
element.

The following result is the key for computing ppd-stingray 
elements from ppd pre-stingray candidates.

\begin{theorem}\label{StingrayOfPrestingray}
Let $G \leq \GL(d,q)$ and let $\ps \in G$ be an $m$-pre-stingray candidate,
i.e.\ we can write $\chi_{\ps}(x) = P_1(x) P_2^{c_2}(x) \dotsc
	P_k^{c_k}(x),$ where $P_1,\dotsc,P_k \in \F_q[x]$ are irreducible, and
$m=\deg(P_1)$ does not divide $\deg(P_i)$ for $i \geq 2$.
Let   $B = p^\beta \prod_{i=2}^k
(q^{\deg(P_i)}-1)$ and $\beta = \lceil \log_p(\max\{c_2,\dotsc,c_k\})
 \rceil$.
Then the following hold:
\begin{enumerate}[label={(\arabic*)}]
\item $\ps^{B}$ has a $(d-m)$-dimensional $1$-eigenspace. 
\item If $x^{(q^m-1)/ \Phi(m,d)} \neq 1$ in $\F_q[x] / \langle P_1(x) \rangle$
	then $\ps$ is a $\ppd(d,q;m)$-element.
\item
If $\ps$ is a $\ppd(d,q;m)$-element, then
$\ps^B$ is a ppd $m$-stingray element, whence 
it acts irreducibly on $W_s$.
\end{enumerate}
\end{theorem}

\begin{proof}
	(1) and (2) are 
	\cite[Remark 3.2]{SR1}. Assertion (3) follows from \cite[p.578]{NeuPrae}
	and \cite[Sec. 4.2]{NCR2}.
\end{proof}

\begin{remark}
\label{FindStingrayElement}
The previous theorem allows us to construct a ppd stingray 
element given a ppd pre-stingray candidate. The 
	ppd pre-stingray candidates we require satisfy $m \in O(\log(d))$.
The probability of finding such ppd
	pre-stingray candidates by random selection has been determined
	Niemeyer and Praeger \cite[Theorem 3.3]{SR1}. 
They show that the proportion of pre-stingray candidates
	in $\GL(d,q)$ which are $\ppd(d,q; m)$ elements
	for $\log(d) \le m \le d/2$ is at least $1-\frac{1}{m}$.
All in all this can be turned into a polynomial-time randomized algorithm \FuncSty{FindStingrayElement} which takes as input a generating set for our group
and a bound $N$ on the maximal number of random elements it may generate before
giving up, and which returns either $\InLineCode{fail}$, or a stringray candidate $g$
together with an MSLP describing it and the new value for $N$, i.e., the input $N$ minus the number of random selections that were used to find $g$.
\end{remark}

\section{\GoingDown\ algorithm} \label{GoingDownSL}
Stingray elements are key to the \GoingDown\ algorithm, as they are
used in the \GoingDown\ basic step to find a subgroup of $\SL(d,q)$
isomorphic to $\SL(d',q)$ quickly, where $d' \leq 4 \lceil \log(d)
\rceil$. Later in this section we study how the \GoingDown\ basic step
can be implemented.
Repeated calls to the \GoingDown\ basic step
yield a descending recognition chain as in
\cref{DescendingRecognitionChain}
\[\SL(4,q) \cong U_k
\leq U_{k-1} \cong \SL(d_{k-1},q)
\leq \dotsc
\leq U_1 \cong \SL(d_1,q)
\leq U_0 = G,\]
where $d_i \leq 4 \lceil \log(d_{i-1}) \rceil$ for $2 \leq i \leq k$.

\subsection*{\GoingDown\ basic step}
Let $G = \SL(d,q)$. In this section we describe and prove the
correctness of one \GoingDown\ basic step. Given a stingray embedded
subgroup $H$ of $G$ with $H \cong \SL(d',q)$  for $d'\le d$ as in
\cref{StandardEmbedded}, the aim of the basic step is to compute a
stingray embedded subgroup $U$ of $G$ with $U \cong \SL(d'',q)$ and $d''
\leq 4 \lceil \log(d') \rceil$. The algorithm to compute the subgroup
$U$ is a randomised algorithm using stingray elements. Therefore this
section relies heavily on \cref{sec:StingrayElements} and the
algorithms presented there.

We start with a sketch of the idea. Let $s_1, s_2 \in \GL(d,q)$ be two
stingray elements, let $\stingb_{s_i} \leq \F_q^d$ be the stingray
bodies, $\dim(\stingb_{s_i}) = n_i$ and $\stingt_{s_i}$ the stingray
tails for $i \in \{1,2\}$, 
where $n_1+n_2 \leq  4 \lceil \log(d) \rceil$. 
As in \cref{StingrayElement}, there are base
change matrices $\BC_i$ such that the stingray element $s_i^{\BC_i}$ is
a matrix with a non-trivial $n_i \times n_i$ block in the upper left
hand corner and 1s on the remaining diagonal. The base change matrix
$\BC_i$ can be computed by appending a basis for the stingray tail
$\stingt_{s_i}$ to a basis for the stingray body $\stingb_{s_i}$ for $i
\in \{1,2\}$.
With high probability the base change matrices
$\BC_1$ and $\BC_2$ are different. However, since $n_1$ and $n_2$ are
relatively small compared to $d$, the elements $s_1$ and $s_2$ must have
a large common $1$-eigenspace, namely $\stingt_{s_1} \cap \stingt_{s_2}$.
In fact we have $\F_q^d = (\stingb_{s_1} + \stingb_{s_2}) \oplus
(\stingt_{s_1} \cap \stingt_{s_2})$ and the two summands are
both invariant under $\langle s_1, s_2 \rangle$.
We consider a base change matrix $\BC_{s_1,s_2} \in
\GL(d,q)$ that arises if we choose a basis for $\stingb_{s_1} +
\stingb_{s_2}$ and append a basis for the large common 1-eigenspace
$\stingt_{s_1} \cap \stingt_{s_2}$ of $s_1$ and $s_2$. Let $n :=
\dim(\stingb_{s_1} + \stingb_{s_2}) = d - \dim(\stingt_{s_1} \cap
\stingt_{s_2})$.
Then the stingray elements $s_1^{\BC_{s_1,s_2}}$ and
$s_2^{\BC_{s_1,s_2}}$ act on $\stingb_{s_1} + \stingb_{s_2}$.
%
With high probability $\stingb_{s_1} \cap
\stingb_{s_2} = \{ 0 \}$, $\langle s_1, s_2 \rangle$ acts irreducibly on
$\stingb_{s_1} + \stingb_{s_2}$ and $\langle s_1, s_2 \rangle \cong
\SL(n,q)$. 

\begin{example}\label{TwoStingrayOneBaseChange}
Let $G = \SL(6,5)$. Then
\[
\scriptsize\arraycolsep=2pt
s_1 := \begin{pmatrix}
1 & 1& 2& 4& 2&3 \\
2 & 4& 0& 4& 4&1 \\
3 & 3& 3& 0& 3&2 \\
4 & 3& 4& 2& 2&3 \\
1 & 4& 0& 2& 3&3 \\
1 & 0& 2& 1& 4&2 \\
\end{pmatrix},
\qquad
s_2 := \begin{pmatrix}
3 &2 &1 &1 &3 &3 \\
3 &3 &0 &4 &0 &3 \\
0 &3 &3 &0 &1 &2 \\
3 &1 &1 &0 &3 &4 \\
0 &2 &3 &0 &0 &3 \\
1 &0 &4 &3 &2 &1 \\
\end{pmatrix}
\]
are two stingray elements with $s_1,s_2 \in G = \SL(6,5)$ and
additionally $\langle s_1, s_2 \rangle \cong \SL(4,5)$. Here
$\stingb_{s_1} = \langle (2,1,1,1,0,0), (4,4,1,0,4,1) \rangle$ and
$\stingb_{s_2} = \langle (3,0,2,4,1,0), (1,4,0,3,0,1) \rangle$, i.e.\
$n_i := \dim(\stingb_{s_i}) = 2$ and $\dim(\stingt_{s_i})=4$.
Moreover, $\stingb_{s_1} \cap \stingb_{s_2} = \{ 0 \}$ thus
$\dim(\stingb_{s_1} + \stingb_{s_2}) = 4$ and so $\dim(\stingt_{s_1}
\cap \stingt_{s_2}) = 2$. Computing a basis for
the common $1$-eigenspace of $s_1$ and $s_2$ yields
$\stingt_{s_1} \cap \stingt_{s_2} = \langle (1,4,0,2,4,0), (0,1,3,0,2,2)
\rangle$ and hence we overall obtain the basis
\begin{align*}
\MB := ( & (2,1,1,1,0,0), (4,4,1,0,4,1),(3,0,2,4,1,0), (1,4,0,3,0,1), \\
         & (1,4,0,2,4,0), (0,1,3,0,2,2))
\end{align*}
for $\F_5^6$. Representing $s_1$ and $s_2$ with respect to this basis $\MB$ gives
\[
\scriptsize\arraycolsep=2pt
{}^{\mathcal{B}} s_1^{\mathcal{B}} = \begin{pmatrix}
 4& 2& \MZero& \MZero& \rvline &\MZero &\MZero \\
 1& 2& \MZero& \MZero&\rvline &\MZero &\MZero \\
 3& 3& 1& 0&\rvline &\MZero &\MZero \\
 4& 2& 0& 1&\rvline &\MZero &\MZero \\
 \hline
 \MZero& \MZero& \MZero& \MZero&\rvline & 1&0 \\
 \MZero& \MZero& \MZero& \MZero&\rvline & 0&1 \\
\end{pmatrix},
\qquad
{}^{\mathcal{B}} s_2^{\mathcal{B}} = \begin{pmatrix}
 1& 0& 0& 0&\rvline &\MZero &\MZero \\
 0& 1& 1& 3&\rvline &\MZero &\MZero \\
 \MZero& \MZero& 3& 2&\rvline &\MZero &\MZero \\
 \MZero& \MZero& 4& 3&\rvline &\MZero &\MZero \\
 \hline
 \MZero& \MZero& \MZero& \MZero&\rvline & 1&0 \\
 \MZero& \MZero& \MZero& \MZero&\rvline & 0&1 \\
\end{pmatrix}.
\]
In order to improve the run-time, it is reasonable to represent the
matrices $s_1$ and $s_2$ as elements of a special linear group
of degree $n_1 + n_2$.
Note that in this example $\SL(4,q) \cong
\langle s_1,s_2 \rangle$ and that $\langle s_1,s_2 \rangle$ is stingray
embedded of degree $4$ in $\SL(6,5)$.
\end{example}

The next algorithm implements the \GoingDown\ basic step, i.e. computes
the next subgroup of a descending recognition chain
\[ \SL(4,q) \cong U_k \leq U_{k-1} \cong \SL(d_{k-1},q)
\leq \dotsc \leq U_1 \cong \SL(d_1,q) \leq U_0 = G, \]
where $d_i \leq 4 \lceil \log(d_{i-1}) \rceil$ for $2 \leq i \leq k$.

\begin{algorithm}[H]
\caption{GoingDownBasicStep}
\label{GoingDownBasicStep}
\small{
{\nlnonumber
\IOKw{
\begin{tabular}{@{}ll}
        \KwIn{}& {\InputT} {\normalfont $\MyIn{d_1} \in \N$ with $\MyIn{d_1}  > 4$}\\
        &{\InputT} {\normalfont $\MyIn{\langle X \rangle = G} \leq \GL(d,q)$ with $\MyIn{G} \cong \SL(\MyIn{d_1},q)$} \\
        &{\InputT} {\normalfont $\MyIn{\BC} \in \GL(d,q)$ a base change matrix such that $\MyIn{G}^{\MyIn{\BC}}$ is stingray embedded} \\
        &{\normalfont in $\MyIn{\GL(d,q)}^{\MyIn{\BC}}$} \\
        &{\InputT} {\normalfont $\MyIn{\MN} \in \N$} \\
        \KwOut{}& \MyFail{fail} OR $\MyOut{(d_2,U,\BC',\slp,\MN')}$ where \\
        &{\OutputT} {\normalfont $\MyOut{d_2} \in \N$ with $4 \leq \MyOut{d_2} \leq 4 \lceil \log(\MyIn{d_1}) \rceil$,} \\
        &{\OutputT} {\normalfont $\MyOut{U} \leq \MyIn{G}$ with $\MyOut{U} \cong \SL(\MyOut{d_2},q)$,} \\
        & {\OutputT} {\normalfont $\MyOut{\BC'}$ is a base change matrix such that $\MyOut{U}^{\MyOut{\BC'}}$ is stingray embedded in $\MyIn{G}^{\MyOut{\BC}}$,} \\
        &{\OutputT} {\normalfont $\MyOut{\slp}$ is an MSLP from $\MyIn{X}$ to generators of $\MyOut{U}$ and} \\
        &{\OutputT} {\normalfont $\MyOut{\MN'} \in \N$ where $\MyIn{\MN}-\MyOut{\MN'}$ is the number of random selections that were used} \\
\end{tabular}
} \\
}

\nlnonumber
\LinesNumbered
\setcounter{AlgoLine}{0}
\Begin($\FuncSty{GoingDownBasicStep} {(} \MyInT{G,\ d_1,\ \BC,\ \MN} {)}$)
{
	\While(\tcp*[f]{\scriptsize{\cref{GlobalCounter}}}){$\MN > 0$}{
		$(s_1,\slp_1,\MN) \asgn$ \algoref{FindStingrayElement}($G^\BC$, $d_1$, $\MN$) \tcp*[r]{\scriptsize{\cref{OutputCheck}}}
		$\stingb_{s_1} \asgn$ \FuncSty{Image}($s_1 - I_{d_1}$) \tcp*[r]{\scriptsize{compute stingray body}}
		\Repeat{$\stingb_{s_1}\cap\stingb_{s_2}=\{0\}$}{
      		$(s_2,\slp_2,\MN) \asgn$ \algoref{FindStingrayElement}($G^\BC$, $d_1$, $\MN$) \tcp*[r]{\scriptsize{\cref{OutputCheck}}}
      		$\stingb_{s_2} \asgn$ \FuncSty{Image}($s_2 - I_{d_1}$) \tcp*[r]{\scriptsize{compute stingray body}}
    		}
    		$d_2 \asgn \dim(\stingb_{s_1})+\dim(\stingb_{s_2})$ \;
    		\uIf(\tcp*[f]{\scriptsize{Using a naming algorithm, see \cref{Naming}}}){$\langle s_1, s_2 \rangle \cong \SL(d_2,\ q)$}{
    			  $\slp \asgn$ an MSLP from $X$ to $(s_1,s_2)$ using $\slp_1$ and $\slp_2$ \;
    			  $\BC_n$ $\asgn$ base change to concatenation of $\FuncSty{Basis}(\stingb_{s_1} + \stingb_{s_2})$ and $\FuncSty{Basis}(\stingt_{s_1} \cap \stingt_{s_2})$\;
			  \Return $\MyOutT{( \langle \operatorname{diag}(s_1,I_{d-d_1}), \operatorname{diag}(s_2,I_{d-d_1}) \rangle,d_2,\operatorname{diag}(\BC_n,I_{d-d_1})\BC,\slp,\MN)}$ \;
    		}
	}
	\Return $\MyOutTT{\fail}$ \;
}
}
\end{algorithm}

\begin{remark}
\begin{enumerate}[label={(\arabic*)}]
\item The input parameter $\MN$ of \algoref{GoingDownBasicStep} is used
   to control the maximal number of random elements chosen. Even though
   \algoref{GoingDownBasicStep} does not compute random elements itself,
   \algoref{FindStingrayElement} is used which is a randomised algorithm to
   compute stingray elements, see \cref{sec:StingrayElements}.
   \algoref{GoingDownBasicStep} is a one-sided Monte Carlo algorithm and
   its success depends on the control parameter $\MN$. Thus if $\MN$ is
   too small, then the algorithm can fail. How $\MN$ has to be chosen in
   order to get a specific success probability $\epsilon \in (0,1)$
   will be subject of a future publication.
\item The condition of the $\InLineCode{if}$ statement in line 9 of
    \algoref{GoingDownBasicStep}, which is $\langle s_1, s_2 \rangle \cong
    \SL(d_2,\ q)$, can be verified using a naming algorithm as in \cref{Naming}.

\item
There is an immediate refinement of this algorithm: If a pair $s_1,s_2$
generating a group $U$ isomorphic to $\SL(d_2,q)$ has been found, then this
isomorphism can be made effective using the action of $U$ on
$W_{s_1}+W_{s_2}$: since $U$ is stingray embedded we can just project onto the
stingray body. Then the next \GoingDown\ step can work with much smaller
matrices of degree $d_2$ instead of degree $d$.
This improves the running time asymptotically while the length and
memory quota of the output MSLPs remain the same.

\end{enumerate}
\end{remark}

We now establish the correctness of \algoref{GoingDownBasicStep} and that
it terminates using at most $\MN$ random elements.

\begin{theorem}\label{GoingDownToDim4StepProof}
\algoref{GoingDownBasicStep} terminates using at most $\MN$ random
selections and works correctly.
\end{theorem}

\begin{proof}
We start by proving that the algorithm terminates.
\algoref{GoingDownBasicStep} contains two loops which start in line 1 and
in line 4. The loop starting in line 1 terminates if $\MN \leq 0$. For every
computation of a random element, $\MN$ is decreased and every call of
\algoref{FindStingrayElement} requires at least one computation of a random
element. Therefore, the statement $\MN \leq 0$ becomes true after at
most $\MN$ executions of \algoref{FindStingrayElement}. In this case
the subroutine ends by returning $\InLineCode{fail}$. By \cref{OutputCheck}
the loop starting in line 4 also returns $\InLineCode{fail}$ if $\MN \leq 0$ by
\algoref{FindStingrayElement} in line 5.

The correctness is clear since the algorithm either returns
$\InLineCode{fail}$ or a subgroup which is isomorphic to $\SL(d_2,q)$ as
this is verified in line 9. We only have to prove that $4 \leq d_2 \leq
4 \lceil \log(d_1) \rceil$. The non-trivial and irreducible subspace of
the stingray element returned by \algoref{GoingDownBasicStep} has
dimension bounded by $2$ and $2 \lceil \log(d_1) \rceil$. Since $d_2$ is
the sum of two such integers it is bounded by $4$ and $4 \lceil
\log(d_1) \rceil$.
\end{proof}

\begin{remark}\label{StingrayNotToDimension2}
Algorithm \GoingDown\ terminates in a group isomorphic to $\SL(4,q)$.
	We now explain, why the method cannot be extended to descend to $\SL(2,q)$.
Note that stingray elements of special linear groups with a
1-dimensional stingray body are equal to the identity matrix. To see
this, apply a base change to obtain a block matrix structure as in
\cref{ex:StingrayElement} where one block visualizes the stingray tail
which is an identity matrix of size $(d-1) \times (d-1)$ and the other
block visualizes the stingray body which has to be $(1) \in \F_q^{1
\times 1}$ since the stingray element is contained in $\SL(d,q)$.

Let $G = \SL(4,q)$ and suppose we search for two stingray elements $s_1,
s_2 \in G$ with $H = \langle s_1, s_2 \rangle \cong \SL(2,q)$. One idea
to overcome this problem
could be to use the \GoingDown\ basic step, i.e.\ to search for two
stingray elements with 1-dimensional stingray bodies which also
intersect trivially and verify $\langle s_1, s_2 \rangle \cong
\SL(2,q)$. As we have noted, stingray elements with 1-dimensional
stingray body are equal to the identity matrix. Therefore, it is not
possible that $\langle s_1, s_2 \rangle \cong \SL(2,q)$.

Another idea is to compute two stingray elements $s_1, s_2 \in G$ which
have the same 2-dimensional stingray body. Note that there are
\[
\binom{4}{2}_q
= \frac{(q^4-1)(q^4-q)}{(q^2-1)(q^2-q)}
= (q^2+1)(q^2+q+1)
> q^4
\]
different 2-dimensional subspaces of $\F_q^4$. Therefore, the
probability that $s_1$ and $s_2$ have the same 2-dimensional stingray
body is less than $1/q^4$. For example, if $q = 121$, then the
probability that $s_1$ and $s_2$ have the same 2-dimensional stingray
body is $0.0000000046$ which makes finding such elements impractical.
\end{remark}

\subsection*{Combining \GoingDown\ basic steps}
Let $G = \SL(d,q)$. By repeatedly calling \algoref{GoingDownBasicStep},
we now have all the methods needed to construct a descending recognition
chain as in \cref{DescendingRecognitionChain} of special linear groups
as
\[ \SL(4,q) \cong U_k \leq U_{k-1} \cong \SL(d_{k-1},q)
\leq \dotsc \leq U_1 \cong \SL(d_1,q) \leq U_0 = G, \]
where $d_i \leq 4 \lceil \log(d_{i-1}) \rceil$ for $1 \leq i \leq k$.
Observe that the chain reaches a $\SL(4,q)$ very rapidly as the basic
step  reduces the dimension logarithmically. The overall number of
\GoingDown\ basic steps required is given by the iterated logarithm
$\log^*(d)$.

\begin{remark}
Note that all state-of-the-art algorithms require a lot more steps. The
currently best algorithm is an algorithm by Dietrich, Lübeck, Leedham-Green and O'Brien
\cite{CREVEN,CRODD}, in the following referred to as DLLO algorithm. The
DLLO algorithm has two variations from which one is based on a
``splitting'' step (called ``One'' in \cite{CRODD}) and the other on a
``conjugating'' step (called ``Two'' in \cite{CRODD}). In practice the
variation of the DLLO algorithm with the ``splitting'' step is used. In
each call to the ``splitting'' step of the DLLO algorithm two subgroups
of half of the input dimension are computed in the best case. Moreover, the
DLLO algorithm has to be applied to both computed subgroups. The
\GoingDown\ algorithm in this publication avoids ``splitting'' by
constructing only one subgroup which is isomorphic to $\SL(4,q)$.
It should be mentioned that the DLLO algorithm is also applicable for
non-natural irreducible matrix representations and for black box
settings and, hence, does not exploit properties of the natural
representation as we do in our \GoingUp\ algorithms which in turn has
an enormous impact on the design of the overall algorithm.
\end{remark}

\begin{remark}
Recall the descending recognition chain
\[ \SL(4,q) \cong U_k \leq U_{k-1} \cong \SL(d_{k-1},q)
\leq \dotsc \leq U_1 \cong \SL(d_1,q) \leq U_0 = G, \]
which is computed by the \GoingDown\ algorithm. For that it uses
\algoref{GoingDownBasicStep} which verifies that a group generated by two
stingray elements is isomorphic to a special linear group using naming
algorithms. Note that the number of calls to naming algorithms can be
reduced by slight variations of the algorithms, e.g.\
\begin{itemize}
\item Always use a naming algorithm, i.e.\ verify that $U_i \cong
    \SL(d_i,q)$ for all $i \in \{1,\dotsc,k \}$,
\item Only use a naming algorithm until the first subgroup is computed,
    i.e.\ only verify that $U_1 \cong \SL(d_1,q)$,
\item Never use a naming algorithm and restart from the input group if a
    maximal number of random selection is reached, i.e.\ restart from
    $U_0 = G$,
\item Never use a naming algorithm but instead of restarting from the
    input group backtrack one step, i.e.\ if a maximal number of tries is
    reached in the process of computing $U_{i+1}$ as a subgroup of $U_i$,
    then restart from $U_{i-1}$ and compute another subgroup $U_i$.
\end{itemize}
Practical tests of implementations could be used to compare the impact
on the running time which is not done in this publication. In the current
implementation of the $\GoingDown$ algorithm in \GAP\ we are using
the third strategy, i.e.\ we never use a naming algorithm and restart 
from the input group if a maximal number of random selection is reached 
which works fine in practice.
\end{remark}

\section{\GoingUp\ algorithm} \label{GoingUpSL}
In this section we assume that we have found a stingray embedded
subgroup $H$ of $G^\BC$ for a known base change matrix $\BC \in
\GL(d,q)$, i.e.\ $H \leq \langle X^\BC \rangle = G^\BC = \SL(d,q)$,
with $H \cong \SL(2,q)$ and an MSLP from $X$ to the standard generators
of $H$. Starting from this setting, we describe an algorithm to compute
standard generators of $G$. Similar to the \GoingDown\ algorithm we
present a \GoingUp\ step and afterwards use the \GoingUp\ step
repeatedly. In this publication we present two versions of a \GoingUp\
step, both of which are randomised algorithms.

The \GoingUp\ step of this section uses linear algebra and relies on
completely new ideas. On the one hand, every computation of this
solution can be performed extremely fast in the natural representations
of special linear groups which results in a very efficient \GoingUp\
step. On the other hand, the computations require that the given
representation of a special linear group is the natural representation
and the \GoingUp\ step of this section cannot easily be modified for
non-natural representations. Using the \GoingUp\ step repeatedly yields
an ascending recognition chain
\[
H = H_{\GUI{0}} \leq H_{\GUI{1}} \dotsc \leq H_{\GUI{\ell-1}} \leq H_{\GUI{\ell}} = G.
\]
as described in \cref{AscendingRecognitionChain} of \cref{OutlineOfTheAlgorithm}.

An alternative solution for \GoingUp\ compared to the one presented in this
publication is based on ideas of \cite{CRODD} and makes use of involution
centralizers.
The advantage of the algorithm we present here is that it is much faster when
applied to a special linear group in its natural representation However, this
comes at the price that the MSLPs for the standard generators of $G$ are
longer and that it cannot be used in non-natural settings. In contrast the
alternative solution is applicable to all classical groups in odd
characteristic. We will present it in a forthcoming publication.

\subsection{Overview of the \GoingUp\ step}
We start this section by stating the main theorem and giving a
description of the \GoingUp\ step in detail. The theorem is proved by
the correctness of the presented \GoingUp\ step of this section. The
algorithm consists of seven $\steps$. One of these $\steps$ is
randomised while the others are deterministic. Since one $\step$ is
randomised, the \GoingUp\ step is also randomised. The \GoingUp\ step
algorithm is presented as \algoref{GoingUpStep} in \cref{ssec:GoingUpStep}.
In \cref{sec:GoingUp} we call the \GoingUp\ step repeatedly to
construct standard generators of the input group $G$ which results in
the final \algoref{GoingUp}.

We start with our hypothesis on the setting of this section and then
state the main theorem. Afterwards, the theorem is divided into seven
$\steps$ which we investigate in more detail and prove their correctness
in the course of this section.

\begin{hypothesis}\label{VariablesGoingUpLinearAlgebra}
For the remainder of this section we assume $\langle X \rangle = G =
\SL(d,q)$ containing a stingray embedded subgroup $H \leq \langle X^\BC
\rangle = G^\BC$ with $H \cong \SL(n,q)$ for $n < d$ and for a known
base change matrix $\BC \in \GL(d,q)$. Moreover, standard generators
$Y_n$ of $H$ are given as words in $X$. Let $V = \F_q^d$ and suppose
that $\MB = (v_1,\dotsc,v_d)$ is a basis of $V$ and let $V_n = \langle
v_1,\dotsc ,v_n \rangle$ and $F_{d-n} = \langle v_{n+1}, \dotsc, v_d
\rangle$. We assume that $H$ acts on $V_n$ as
$\SL(n,q)$ and that $H$ fixes $F_{d-n}$ pointwise. Recall that
$(\prel_1,\dotsc,\prel_f)$ is an $\F_p$-basis of $\F_q$.
\end{hypothesis}

The main theorem of the \GoingUp\ step is the following.

\begin{theorem}\label{DoublingTheDimension}
Let $X \subseteq \SL(d,q)$ such that $\langle X \rangle= G = \SL(d,q)$
and let $2 \leq n < d$ with $n = 2$ or $n$ odd, let $\BC \in \GL(d,q)$
be a base change matrix. Let $Y_n^\BC$ be a set of standard generators
for the subgroup $\SL(n,q)$ stingray embedded into $G^\BC$. Furthermore,
let $\slp$ be a straight-line program from $X$ to $Y_n$ and let $n' :=
\min\{2n - 1, d \}$.

Then there is an algorithm that computes a base change matrix $\BC' \in
\GL(d,q)$ together with a straight-line program $\slp'$ from $X$ to a
set $Y_{n'}$, which is a set of standard generators for $\SL(n',q)$ and
$Y_{n'}^{\BC'}$ is stingray embedded in $G^{\BC'}$.
\end{theorem}

We prove \cref{DoublingTheDimension} by stating an algorithm. The
algorithm consists of seven $\steps$, called \ref{SL1} to \ref{SL7}, which
are discussed and proven in the remainder of this section.

\begin{remark}\label{IdeaOfGoingUpBasicStepSL}
The general idea of the algorithm is the following. Let $H = \SL(V_n)
\cong \SL(n,q)$ stingray embedded in $G^{\BC}$ and $G = \SL(d,q)$. Given
standard generators $Y_n$ for $H$, we construct an element $\cM \in
G^{\BC}$ with the following properties:
\begin{enumerate}[label={(C\arabic*)},leftmargin=*]
\item\label{C1} $\dim(V_n + V_n \cM) = n'$,
\item\label{C2} if $n' < d$, then $\dim(\Fix(H) + \Fix(\cM))
= \dim(F_{d-n} + \Fix(\cM)) = d$.
\end{enumerate}
Note that $H$ acts on $V_n$ as $\SL(V_n) \cong \SL(n,q)$ and similarly
$H^{\cM}$ acts on $V_n \cM$ as $\SL(V_n \cM) \cong \SL(n,q)$. Then we
have
\[ \dim(V_n \cap V_n \cM)
= \dim(V_n) + \dim(V_n \cM) - \dim(V_n + V_n \cM)
= 2n - n'
= \max \{ 1, 2n - d \}. \]
In \cref{ssec:GoingUpPhase2} we formulate an additional property \ref{C3}
which cannot be defined at this point. If $\cM$ does not satisfy \ref{C3},
then we construct a new element $\cM \in G^{\BC}$. In the remainder of
this section we prove that this setup allows us to choose a basis of the
$n'$-dimensional subspace $V_{n'} =V_n+V_n \cM$ of $V$ carefully such
that we can use $\cM$-conjugates of certain transvections in $H$ to
assemble standard generators for $\SL(n',q)$ with respect to the new
basis. As a result this allows us to conclude that $\langle H,H^{\cM}
\rangle$ is indeed isomorphic to $\SL(n', q)$. When $n' \leq d$ we have
roughly doubled the degree from $n$ to $2n-1$ using $\cM$.
\end{remark}

\begin{definition}\label{DoublingElement}
Assume the setting as described in \cref{VariablesGoingUpLinearAlgebra} and let $\cM \in G^{\BC}$.
\begin{enumerate}[label={(\arabic*)}]
\item If $\cM$ satisfies \ref{C1} and \ref{C2} of
\cref{IdeaOfGoingUpBasicStepSL}, then $\cM$ is a \Defn{weak doubling
element} with respect to $H$.
\item If $\cM$ is a weak doubling element with respect to $H$ and
additionally fixes $v_n$, then $\cM$ is a \Defn{doubling element} with
respect to $H$.
\end{enumerate}
When the context is clear, an weak doubling element and doubling element
with respect to $H$ are only denoted by a weak doubling element and
doubling element.
\end{definition}

\begin{remark}\label{TasksOfPhases}
The next sections can roughly be summarized as follows.
\begin{enumerate}[label={(\arabic*)}]
\item \textbf{Construction of a doubling element:} Construct an element
$\cM \in G^{\BC}$ fulfilling the properties \ref{C1} and \ref{C2} and fixing
$v_n$, i.e.\ a doubling element as in \cref{DoublingElement}. This is
achieved by random selection of elements of $G^{\BC}$ and discussed in
\cref{ssec:GoingUpPhase1}.
\item \textbf{Construction of a new base change matrix:} Construct a new
base change matrix $\BC'$ such that $\langle H,H^{\cM} \rangle^{\BC'}$
is stingray embedded in $G^{\BC \BC'}$ which is possible because of the
properties \ref{C1} and \ref{C2}. The computation of $\BC'$ is deterministic
and discussed in \cref{ssec:GoingUpPhase2}. In this section we also
formulate an additional property \ref{C3} which can only be given after
constructing $\BC'$. If $\cM$ does not satisfy \ref{C3}, then we start
over from (1).
\item \textbf{Construction of transvections and standard generators:} If
(2) is successful, then we can conclude that $\langle H,H^{\cM}
\rangle^{\BC'} \cong \SL(n',q)$. Hence, we construct transvections and
standard generators for $\langle H,H^{\cM} \rangle^{\BC'}$. Note that we
only need to compute permutation matrices corresponding to the $n'$ and
$(n'-1)$ cycle of \cref{StandardGeneratorsSL} as the transvections of
the standard generators of $H$ are also the transvections of the
standard generators of $\langle H,H^{\cM} \rangle^{\BC'}$. The
computations performed in (3) are deterministic and discussed in
\cref{ssec:GoingUpPhase3}.
\end{enumerate}
\end{remark}

In each section the task described in \cref{TasksOfPhases} is divided
into more $\steps$ which are discussed in detail in the corresponding
section and labelled as \ref{SL1} to \ref{SL7}. Lastly, in
\cref{ssec:GoingUpStep} the $\steps$ \ref{SL1} to \ref{SL7} are combined
into a single algorithm for the \GoingUp\ step. Proving its correctness
yields the proof for \cref{DoublingTheDimension}.

\subsection{Construction of a dimension doubling element} \label{ssec:GoingUpPhase1}

The goal in this section is the construction of a doubling element,
i.e.\ an element $\cM \in G^{\BC}$ fulfilling the properties \ref{C1} and
\ref{C2} and fixing $v_n$, leading to \algoref{ComputeDoublingElement}.
Recall from \cref{VariablesGoingUpLinearAlgebra} the setting
for this section. We give a condensed version of what is achieved.

\begin{remark}\label{StepOneToThreeOfSL}\ 
\begin{enumerate}[label={(SL\arabic*)},leftmargin=*]
\item[\SLFi] Construct an element $ \GUE \in H $ which has a fixed space
of dimension $ d - n + 1 $. This can be done easily as we have standard
generators $Y_n$ for $H$.
\item[\SLSe] Choose random elements $ \cE \in G^{\BC} $ until $ \cT :=
\GUE^{\cE} $ 
is a weak doubling element.
\item[\SLTh] Find a conjugate $\cM \in G^{\BC}$ of $\cT$ which 
is a doubling element.
\end{enumerate}
\end{remark}

Note that the \GoingUp\ algorithm is randomised as in \ref{SL2} we search
for random $\cE \in G^{\BC}$ such that $\cT := \GUE^{\cE}$ fulfils the
two properties \ref{C1} and \ref{C2} of \cref{IdeaOfGoingUpBasicStepSL}. We
do not analyse the proportion of usable elements in this section,
only prove the correctness. 

The following lemma shows how \ref{SL1} is carried out in the solution of
this publication. Note that there are multiple other choices for $\GUE
\in H$ having a fixed space of dimension $d-n+1$. Recall the permutation
matrices $\PME$ and $\PMZ$ from \cref{StandardGeneratorsSL} which are
permutation matrices corresponding to $n$ and $n-1$ cycles respectively.

\begin{lemma}\label{Step1}
An element of $H$ which has a fixed space of dimension $d-n + 1$ is given by
\[
\GUE :=
\begin{cases}
E_{1,2}(1), &\text{ if } n = 2, \\
\PME, &\text{ if } n > 2 \text{ and } p \text{ is even}, \\
\PMZ, &\text{ if } n > 2 \text{ and } p \text{ is odd}. \\
\end{cases}
\]
\end{lemma}

\begin{proof}
If $n = 2$, then we take $\GUE = E_{1,2}(1)$. Its fixed space is
$\langle v_2,\dotsc,v_d \rangle$ and thus has dimension $d-1 = d -2 +1 =
d - n + 1$. If $n > 2$, then by assumption in
\cref{DoublingTheDimension} $n$ is odd. For $\GUE$ we take $\PMZ$ a $(n
- 1)$-cycle if the characteristic $p$ is odd and $\PME$ a $n$-cycle if
$p = 2$. The fixed space is $\langle v_1,v_{n+1},v_{n+2},\dotsc,v_d
\rangle$ in the first case and $\langle v_1 + v_2 + \dotsc
+v_n,v_{n+1},\dotsc,v_d \rangle$ in the second case.
\end{proof}

The element $\GUE$ is “constructed” as an MSLP in the given standard
generators of $\SL(n, q)$, but can also be written explicitly as a
matrix of $\langle X \rangle^\BC$ at a cost of $O(d^2)$.

Now that $\GUE$ has a fixed space of dimension $d-n+1$, we choose a
random $\cE \in G^\BC$ and check whether $\cT = \GUE^{\cE}$ fulfils
the properties \ref{C1} and \ref{C2} of \cref{IdeaOfGoingUpBasicStepSL}.
We repeat this until $\cT$ has the required properties.

\begin{remark}
Verifying \ref{C1} and \ref{C2} of \cref{IdeaOfGoingUpBasicStepSL} is
one of the few places in \GoingUp\ step where we must multiply $(d
\times d)$-matrices, since $\cE$ is created as a matrix in the same
input basis as $X$, and we must construct the matrix of $\cE$ and $\cT$
as an element of $\langle X \rangle^\BC$. This is necessary to check
\ref{C1} and \ref{C2}.
\end{remark}

\begin{lemma}\label{SomeResultsOnVnc}
Let $\GUE$ be as in \cref{Step1}.
\begin{enumerate}[label={(\arabic*)}]
\item If $\cT := \GUE^{\cE}$ is a weak doubling element, then $\dim(V_n
\cap \Fix(\cT))\geq 1$ and $\dim(V_n \cap \Fix(\cT)) = 1$ if $n' < d$.
\item Then there is a $\cE \in G^\BC$ such that $\cT := \GUE^{\cE}$ is a
weak doubling element, i.e.\ $\cT$ satisfies \ref{C1} and \ref{C2}.
\item If $\cT := \GUE^{\cE}$ is a weak doubling element, then $V_{n'}$
is invariant under the action of $\cT$.
\item If $\cT := \GUE^{\cE}$ is a weak doubling element and $n' < d$,
then $\dim(F_{d-n} \cap \Fix(\cT)) = d-n'$.
\end{enumerate}
\end{lemma}

\begin{proof}\
\begin{enumerate}[label={(\arabic*)}]
\item Since $\cT$ and $\GUE$ are conjugate we know that $\dim(\Fix(\cT))
    = \dim(\Fix(\GUE)) = d-n+1$. Note that $\dim(V_n + \Fix(\cT)) \leq
    \dim(V) = d$. Hence,
    \begin{align*}
    \dim(V_n \cap \Fix(\cT))
    &= \dim(V_n) + \dim(\Fix(\cT)) - \dim(V_n + \Fix(\cT)) \\
    &= n+(d-n+1)-\dim(V_n + \Fix(\cT)) \\
    &= d+1 -\dim(V_n + \Fix(\cT)) \geq 1.
    \end{align*}
    Now let $n' < d$, i.e. $n' = 2n - 1$. Then using \ref{C1} of
    \cref{IdeaOfGoingUpBasicStepSL} it follows that
    \begin{align*}
    \dim(V_n \cap V_n \cT)
    &= \dim(V_n) + \dim(V_n \cT) - \dim(V_n + V_n \cT) \\
    &= n+n - n'\\
    &= n + n - (2n -1) = 1.
    \end{align*}
    Notice that $V_n \cap \Fix(\cT) \subseteq V_n \cT$ and thus
    \[ \dim(V_n \cap \Fix(\cT)) \leq \dim(V_n \cap V_n \cT) = 1. \]
    Since $\dim(V_n \cap \Fix(\cT)) \geq 1$, the result follows.
\item Note that an element fulfilling the properties \ref{C1} and \ref{C2} exists as 
for example matrices of the form $\cE =\operatorname{diag}(\tilde{a},I_{d-n'})$ with 
\[
\begin{tikzpicture}[line cap=round,line
join=round,x=1cm,y=1cm,baseline={-0cm-0.5*height("$=$")},mymatrixenv]
\matrix[mymatrix] (m)  {
  0 & {\hdots} & 1 & 0& 0 & {\hdots} & 1\\
  \vdots & {\ddots} & {\vdots} & {\vdots} & {\vdots} & {\ddots} & {\vdots} \\
  1 & \hdots & 0 & 0 & 1 & \hdots & 0 \\
  0& {\hdots} & 0 & 1 & 0 & {\hdots} & 0 \\
  0 & {\hdots} & 0 & 0 & 1 & {\hdots} &0 \\
 {\vdots} &{\ddots} & {\vdots} & {\vdots} & {\vdots} & {\ddots} & {\vdots} \\
  0 & {\hdots} & 0 & 0 & 0 & {\hdots} & 1 \\
};
\mymatrixbraceleft{1}{3}{$n-1$}
\mymatrixbraceleft{4}{4}{$1$}
\mymatrixbraceleft{5}{7}{$n'-n$}
\mymatrixbracetop{1}{3}{$n-1$}
\mymatrixbracetop{4}{4}{$1$}
\mymatrixbracetop{5}{7}{$n'-n$}
\end{tikzpicture}
\eqqcolon \tilde{a}
\]
can be chosen. It is clear that $\cT = \GUE^{\cE}$ fulfils \ref{C1}. Note 
that $e_1 + e_2 + \dotsc + e_n \in V_n \cap \Fix(\cT)$ and thus \ref{C2}
also holds as $\dim(F_{d-n} + \Fix(\cM)) = d$ is equivalent to 
$\dim(F_{d-n} \cap \Fix(\cM)) \geq 1$ by 1).
\item If $n' = d$, then the statement is trivial. Let us assume that $n'
    = 2n-1 < d$. We construct a basis of $V_{n'}$ which shows that $V_{n'}$
    is invariant under the action of $\cT$. A basis of $V_{n'}$ consists of
    $n'$ elements since $\dim(V_{n'}) = n'$ by \ref{C1}. First notice that
    using (2) it follows that
    \begin{align*}
    \dim(V_n + \Fix(\cT))
    &= \dim(V_n) + \dim(\Fix(\cT)) - \dim(V_n \cap \Fix(\cT)) \\
    &= n + (d-n+1) - 1 = d.
    \end{align*}
    Therefore, $V_n + \Fix(\cT) = V$. Since $V_n \leq V_{n'}$, it follows
    that $V_{n'} + \Fix(\cT) = V$ which implies that
    \begin{align*}
    \dim(V_{n'} \cap \Fix(\cT))
    &= \dim(V_{n'}) + \dim(\Fix(\cT)) - \dim(V_{n'} + \Fix(\cT)) \\
    &= (2n-1) + (d-n+1) -d = n.
    \end{align*}
    Since $\dim(V_n \cap (V_{n'} \cap \Fix(\cT)) = \dim(V_n \cap \Fix(\cT))
    = 1$, it follows that
    \begin{align*}
    \dim(V_n + (V_{n'} \cap \Fix(\cT)) )
    &= \dim(V_n) + \dim(V_{n'} \cap \Fix(\cT)) - \dim(V_n \cap (V_{n'} \cap \Fix(\cT))) \\
    &= n + n - 1 = 2n -1 = n'.
    \end{align*}
    Since $V_n + (V_{n'} \cap \Fix(\cT)) \leq V_{n'}$, it follows that $V_n
    + (V_{n'} \cap \Fix(\cT)) = V_{n'}$. Therefore, we can choose a basis of
    $V_{n'}$ as follows:
    \begin{itemize}
    \item Choose a non-zero vector $v_1 \in V_n \cap \Fix(\cT)$.
    \item Select $n-1$ vectors $v_2, \dotsc, v_n \in V_n$ to extend this to
    a basis of $V_n$.
    \item Choose $n-1$ vectors from $V_{n'} \cap \Fix(\cT)$ to extend this
    to a basis of $V_{n'}$.
    \end{itemize}
    Clearly this basis is invariant under the action of $\cT$ as either $v_i
    \cT = v_i$ or $v_i \cT \in V_n \cT \leq V_{n'}$.
\item If $n' < d$, then $\dim(F_{d-n} + \Fix(\cT)) = d$ by \ref{C2}.
Hence, as claimed,
\begin{align*}
\dim(F_{d-n} \cap \Fix(\cT))
&= \dim(F_{d-n}) + \dim(\Fix(\cT)) - \dim(F_{d-n} + \Fix(\cT)) \\
&= (d-n) + (d-n+1) - d \\
&= d-2n+1 = d-n'. \qedhere
\end{align*}
\end{enumerate}
\end{proof}

\begin{remark}
If $n' = d < 2n-1$, then it is possible that $\dim(V_n \cap \Fix(\cT)) >
1$ which is important in \cref{Step4SL}.
\end{remark}

\begin{corollary}
$V_n + V_n \cT = V_n + V_n \cT^{-1}$.
\end{corollary}

\begin{proof}
Note that $V_n + V_n \cT = V_{n'}
\stackrel{\ref{SomeResultsOnVnc}~(3)}{=} V_{n'} \cT^{-1}
= (V_n + V_n \cT) \cT^{-1} = 
V_n \cT^{-1} + V_n$.
\end{proof}

\begin{remark}
Note that there is one more condition \ref{C3} which we can only be
formulate in \cref{ssec:GoingUpPhase2}. If that condition is not
satisfied, then we restart from \ref{SL2} and try another element $\cE$.
\end{remark}

Our next goal is to construct a conjugate $\cM$ of $\cT$ such that $\cM$
is a doubling element. We seek an element $\GUFE \in H$ with $v_n \GUFE
= v$, where $0 \neq v \in V_n \cap \Fix(\cT)$, and write $\GUFE$ as a
word in the standard generators $Y_n$ and finally compute $\cM := \GUFE
\cT \GUFE^{-1}$. It is well-known that $\SL(d,q)$ acts doubly transitive
on the 1-dimensional subspaces of $\F_q^d$ (see e.g.~\cite[Theorem 4.1]{Taylor}).
This implies the existence of
an element $\GUFE \in H$ with the described properties, \cref{GoingUpStepThree} shows that $\cM$ is a doubling element and lastly
we describe how $\GUFE$ can be found. We start by proving that $\GUFE
\in H \cong \SL(n,q)$ exists.

Next we show that $\cM := \GUFE \cT \GUFE^{-1}$ is a doubling element.

\begin{lemma}\label{GoingUpStepThree}
Let $\GUE$ and $G$ and $H$ be as in \cref{Step1} and $\cT :=
\GUE^{\cE}$ a weak doubling element for a random element $\cE \in
G^\BC$. Let $\GUFE \in H$ such that $v_n \GUFE = v$ where $0 \neq v \in
V_n \cap \Fix(\cT)$. Then
\begin{enumerate}[label={(\arabic*)}]
\item $\GUFE \cT \GUFE^{-1}$ fulfills \ref{C1} and \ref{C2}.
\item $v_n \in V_n \cap \Fix(\GUFE \cT \GUFE^{-1})$.
\end{enumerate}
\end{lemma}

\begin{proof}
\begin{enumerate}[label={(\arabic*)}]
\item Since $\GUFE \in H$, it follows that $V_n \GUFE = V_n$ and so
\[ \dim(V_n + V_n \GUFE \cT \GUFE^{-1})
= \dim(V_n \GUFE + V_n \GUFE \cT)
= \dim(V_n + V_n \cT) = n' \]
and \ref{C1} holds. Moreover, $\GUFE F_{d-n} \GUFE^{-1} = F_{d-n}$ (notice
$\GUFE \in H$) and so
\[ \dim(\Fix(\GUFE \cT \GUFE^{-1}) + F_{d-n})
= \dim(\Fix(\GUFE \cT \GUFE^{-1}) + \GUFE F_{d-n} \GUFE^{-1})
= \dim(\Fix(\cT) + F_{d-n}) = d \]
and \ref{C2} holds.
\item $v_n \in V_n$ and $v_n \GUFE \cT \GUFE^{-1} = v \cT \GUFE^{-1} = v
\GUFE^{-1} = v_n$ since $0 \neq v \in V_n \cap \Fix(\cT)$.
\qedhere
\end{enumerate}
\end{proof}

\begin{remark}\label{Step3SL}
We proceed as follows to find $\GUFE$ fulfilling the hypothesis of \cref{GoingUpStepThree}:
\begin{itemize}
\item Let $0 \neq v \in V_n \cap \Fix(\cT)$ and $v = \sum_{j=1}^{n}
\lambda_j v_j$ for $\lambda_j \in \F$. This is possible since $v \in
V_n$ and $(v_1,\dotsc,v_n)$ is a basis of $V_n$.
\item If $\lambda_n \neq 0$, then we multiply $v$ by $\lambda_n^{-1}$
such that without loss of generality we may assume $\lambda_n = 1$. This
allows us to write $\GUFE$ as a product of elements of the form
$E_{n,j}(\prel_i)$ for $1 \leq i \leq f$ and $1 \leq j \leq n-1$ which
are easily expressed in terms of the standard generators $Y_n$ using a rewriting procedure \cite{Costi,MSLP2024}.
\item If $\lambda_n = 0$, then we choose $j$ with $\lambda_j \neq 0$ and
then a $\GUFE' \in H$ with $v_n \GUFE' = vE_{j,n}(1)$ as above
(replacing $v$ with a scalar multiple as needed). Finally we set $\GUFE
:= \GUFE' E_{j,n}(1)^{-1}$.
\end{itemize}
\end{remark}

\begin{remark}\label{NoteToStep3}
Notice that the matrices $E_{1,2}(\prel_i)$ and $E_{2,1}(\prel_i)$ are
given as elements of $\langle X \rangle^\BC$ which contains the stingray
embedded subgroup $H$ isomorphic to $\SL(n,q)$ and standard generators
$Y_n$ can be written as words in $X^\BC$ as described in \cref{VariablesGoingUpLinearAlgebra}.
\end{remark}

\begin{algorithm}[H]
\caption{ComputeDoublingElement}
\label{ComputeDoublingElement}
\small{
{\nlnonumber
\IOKw{
\begin{tabular}{@{}ll}
		\KwIn{}& {\InputT} {\normalfont $\MyIn{\langle X \rangle = G} \leq \SL(d,q)$}\\
        &{\InputT} {\normalfont A base change matrix $\MyIn{\BC} \in \GL(d,q)$ }\\
        &{\InputT} {\normalfont $\SL(n,q) \cong \MyIn{ \langle Y_n \rangle = H} \leq \MyIn{G}^{\MyIn{\BC}}$ stingray embedded and constructively recognised}\\
        &{\InputT} {\normalfont $\MyIn{\MN} \in \N$}\\
        \KwOut{}& \MyFail{fail} OR $\MyOut{(\cM,\slp,\MN')}$ where \\
        & {\OutputT} {\normalfont $\MyOut{\cM} \in \MyIn{G}^{\MyIn{\BC}}$ is a doubling element,} \\
        & {\OutputT} {\normalfont $\MyOut{\slp}$ is an MSLP from $\MyIn{X \cup Y_n}$ to $\MyOut{\cM}$ and} \\
        & {\OutputT} {\normalfont $\MyOut{\MN'} \in \N$ where $\MyIn{\MN}-\MyOut{\MN'}$ is the number of random selections that were used} \\
\end{tabular}
} \\
}

\nlnonumber
\LinesNumbered
\setcounter{AlgoLine}{0}
\Begin($\FuncSty{ComputeDoublingElement} {(} \MyInT{G, H, \BC, \MN} {)}$)
{
  Choose $\GUE \in H$ as described in \cref{Step1} \tcp*[r]{\scriptsize{\ref{SL1}}}
  \Repeat(\tcp*[f]{\scriptsize{Start of \ref{SL2}}}){$\cT$ satisfies \ref{C1} and \ref{C2}}
  {
  	$\MN \asgn \MN - 1$ \;
  	\uIf{$\MN < 0$}
  	{
		\Return $\MyOutTT{\fail}$ \;
  	}
  	$\cT \asgn \GUE^{\cE}$ for random $\cE \in G^{\BC}$ \;
  }
    $\slp_1 \asgn$ MSLP from $X \cup Y_n$ to $\cT$ \;
  $\GUFE \asgn$ as described in \cref{Step3SL} \; $\slp_2 \asgn$ MSLP from $Y_n$ to $\GUFE$ \tcp*[r]{\scriptsize{\ref{SL3}}}
  $\cM \asgn \GUFE \cT \GUFE^{-1}$ \; $\slp \asgn $ Compose $\slp_1$ and $\slp_2$ \;
   \Return $\MyOutT{(\cM,\slp,\MN)}$ \;
}
}
\end{algorithm}

\subsection{Construction of a new base change matrix} \label{ssec:GoingUpPhase2}
The goal of this section is the construction of a new base change matrix
$\BC'$ such that $\langle H,H^{\cM} \rangle^{\BC'}$ is stingray embedded
in $G^{\BC \BC'}$. Everything done in this section is deterministic.
Recall from \cref{VariablesGoingUpLinearAlgebra} the setting
for this section and that $\cM \in G^{\BC}$ is a doubling element, i.e.\
satisfies \ref{C1} and \ref{C2} and fixes $v_n$ as described in
\cref{ssec:GoingUpPhase1}. This section only covers one $\step$ as
described in the next remark.

\begin{remark}\label{StepFourOfSL}\ 
\begin{itemize}
\item[\SLFo] Compute a base change matrix $\BC'$ such that $\langle
H,H^{\cM} \rangle^{\BC'}$ is stingray embedded in $G^{\BC \BC'}$.
\end{itemize}
\end{remark}

The goal of \ref{SL4} is to compute a base change matrix $\BC' \in
\GL(d,q)$ such that $\langle H, H^{\cM} \rangle^{\BC'}$ is stingray
embedded in $G^{\BC \BC'}$ and if $\langle H, H^{\cM} \rangle \cong
\SL(n',q)$, then standard generators of $\langle H, H^{\cM}
\rangle^{\BC'}$ can be constructed using the steps \ref{SL5} to \ref{SL7}
described in \cref{ssec:GoingUpPhase3}. Note that we formulate an
additional condition \ref{C3} in this section which can only be defined
after the construction of $\BC'$. If $\cM$ does not satisfy \ref{C3}, then
we have to chose a new random element $\cE \in G^\BC$ and restart from
$\step$ \ref{SL2}. If $\cM$ satisfies \ref{C3}, then we can conclude in
\cref{ssec:GoingUpPhase3} that $\langle H, H^{\cM} \rangle \cong
\SL(n',q)$ and construct standard generators of $\langle H, H^{\cM}
\rangle$.

\begin{remark}\label{Step4SL}
The base change matrix $\BC'$ is computed by constructing a specific
basis of $\F_q^d$. Let $\pi \colon V \to F_{d-n}$ be the projection map
to the second summand of $V = V_n \oplus F_{d-n}$.
\begin{itemize}
\item The first $n$ vectors of the new basis are equal to the vectors in
the old basis, i.e. $v_i' := v_i$ for $1\leq i\leq n$.
\item The vectors $v_{n +1}', \dotsc , v_{n'}'$ of the new basis are
chosen as a linearly independent subset of the vectors $\pi(v_1
\cM),\dotsc , \pi(v_{n-1} \cM)$. Notice that this is possible since
$\dim(V_n + V_n \cM) = n'$ and $v_{n}' =v_n \in V_n$ is fixed by $\cM$.
If $n' < d$, then we take all the vectors $\pi(v_1 \cM),\dotsc ,
\pi(v_{n-1} \cM)$ and otherwise we choose a linearly independent subset.
\item Finally, in the case $n' < d$, that is, $n' = 2n-1$, we extend
$(v_1' ,\dotsc,v_{n'} ')$ to a basis $\MB' = (v_1' ,\dotsc,v_d' )$ of
$V$ by choosing a basis of $F_{d-n} \cap \Fix(\cM)$ which is possible by
condition \ref{C2}.
\end{itemize}
\end{remark}

The matrix $\BC' \in \GL(d,q)$ is now chosen to be the base change
matrix between $\MB$ and $\MB'$. Note that \ref{C1} and \ref{C2} ensure that
$\langle H,H^{\cM} \rangle$ can be stingray embedded in $G^{\BC}$ using
\cref{Step4SL}.

\begin{remark}
\begin{enumerate}[label={(\arabic*)}]
\item $V_n = \langle v_1',\dotsc,v_n' \rangle$.
\item $V_n + V_n \cM = \langle v_1',\dotsc,v_{n'}' \rangle$.
\item $\langle v_{n+1}',\dotsc,v_d' \rangle \subseteq \Fix(H)$
\item Note that (1) and (3) ensure that for our stingray embedded
subgroup $H \cong \SL(n, q)$ in $G^\BC$ also $H^{\BC'}$ is a stingray
embedded subgroup in $G^{\BC \BC'}$.
\end{enumerate}
\end{remark}

From this point on, we assume that all of our matrices are expressed as
elements of $G^{\BC \BC'}$. For the standard generators of $\SL(n,q)$,
there is nothing to do, while the matrix of $\cM$ must be conjugated by
the base change matrix $\BC'$ which we denote by $\cH := \cM^{\BC'}$.

At this point we formulate the last condition \ref{C3}. If \ref{C3} is
satisfied by $\cH$, then we can conclude that $\langle H, H^{\cM}
\rangle \cong \SL(n',q)$ in \cref{ssec:GoingUpPhase3}. The necessity for
condition \ref{C3} is discussed in more detail in
\cref{ssec:GoingUpPhase3}.

\begin{remark}\label{LastConditionOnCForSL}
The final condition on $\cH$ is the following.
\begin{enumerate}[label={(C\arabic*)},leftmargin=*,start=3]
\item\label{C3} The vectors $\prel_i v_j'$ and $\prel_i v_j' \cH^{-\tr}$
for $1 \leq i \leq f$ and $1 \leq j \leq n-1$ span the subspace $\langle
v_1' ,\dotsc ,v_{n-1}', v_{n+1}' ,\dotsc ,v_{n'}'\rangle$.
\end{enumerate}
\end{remark}

\begin{definition}
Assume the setting as described in \cref{VariablesGoingUpLinearAlgebra}. Let $\cM \in G^{\BC}$ be a doubling
element and $\cH := \cM^{\BC'}$, where $\BC'$ is chosen as in
\cref{Step4SL}. If $\cH$ additionally satisfies \ref{C3}, then $\cH$ is a
\Defn{strong doubling element}.
\end{definition}

Note that \ref{C3} can only be verified after computing the base change
matrix $\BC'$ which is also the reason why this condition is introduced
at this point. If \ref{C3} is not satisfied by $\cH$, then we return to
\ref{SL2} and try another random element $\cE \in G^{\BC}$. Note that
$\langle H, H^{\cM} \rangle^{\BC'} = \langle H, H^{\cH} \rangle $. We
give an algorithm to compute a strong doubling element $\cH \in G^{\BC
\BC'}$ satisfying \ref{C1}, \ref{C2} and \ref{C3} in pseudo code as
\algoref{ComputeStrongDoublingElement} using \algoref{ComputeDoublingElement}
of \cref{ssec:GoingUpPhase3}.

\begin{algorithm}[H]
\caption{ComputeStrongDoublingElement}
\label{ComputeStrongDoublingElement}
\small{
{\nlnonumber
\IOKw{
\begin{tabular}{@{}ll}
		\KwIn{}& {\InputT} {\normalfont $\MyIn{\langle X \rangle = G} \leq \SL(d,q)$}\\
        &{\InputT} {\normalfont A base change matrix $\MyIn{\BC} \in \GL(d,q)$ }\\
        &{\InputT} {\normalfont $\SL(n,q) \cong \MyIn{ \langle Y_n \rangle = H} \leq \MyIn{G}^{\MyIn{\BC}}$ stingray embedded and constructively recognised}\\
        &{\InputT} {\normalfont $\MyIn{\MN} \in \N$}\\
        \KwOut{}& \MyFail{fail} OR $\MyOut{(\cH,\BC',\slp,\MN')}$ where \\
        & {\OutputT} {\normalfont $\MyOut{\cH} \in \MyIn{G}^{\MyIn{\BC}\MyOut{\BC'}}$ is a strong doubling element,} \\
        & {\OutputT} {\normalfont $\MyOut{\BC'} \in \GL(d,q)$ is a base change matrix as in \cref{Step4SL},} \\
        & {\OutputT} {\normalfont $\MyOut{\slp}$ is an MSLP from $\MyIn{X \cup Y_n}$ to $\MyOut{\cH}$ and} \\
        & {\OutputT} {\normalfont $\MyOut{\MN'} \in \N$ where $\MyIn{\MN}-\MyOut{\MN'}$ is the number of random selections that were used} \\
\end{tabular}
} \\
}

\nlnonumber
\LinesNumbered
\setcounter{AlgoLine}{0}
\Begin($\FuncSty{ComputeStrongDoublingElement} {(} \MyInT{G,\BC, H, \MN} {)}$)
{
  \While{$\MN > 0$}
  {
   $(\cM,\slp,\MN) \asgn \algoref{ComputeDoublingElement}(G,\BC,H,\MN)$ \tcp*[r]{\scriptsize{\cref{GlobalCounter,OutputCheck}}}
   $\BC' \asgn $ as described in \cref{Step4SL} \;
   $\cH \asgn \cM^{\BC'}$ \tcp*[r]{\scriptsize{\ref{SL4}}}
   \uIf{the $(n'-n)\times(n-1)$-submatrix $(\cH ^{-1})_{i,j}$ for $n+1 \leq i \leq n'$ and $1 \leq j \leq n-1$ has full rank}
   {
       \Return $\MyOutT{(\cH,\BC',\slp,\MN)}$ \;
   }
   }
   \Return $\MyOutTT{\fail}$ \;
}
}
\end{algorithm}

\subsection{Construction of transvections and standard generators} \label{ssec:GoingUpPhase3}
Given a strong doubling element $\cH \in G^{\BC \BC'}$ as described in
\cref{ssec:GoingUpPhase2} the goal of this section is to conclude that
$\langle H, H^{\cH} \rangle \cong \SL(n',q)$ and the construction of
transvections and standard generators for $\langle H, H^{\cH} \rangle$.
Everything done in this section is deterministic. Recall from \cref{VariablesGoingUpLinearAlgebra} the setting of this section and that
$\cH \in G^{\BC \BC'}$ satisfies \ref{C1}, \ref{C2} and \ref{C3}. Moreover,
note that $\langle H, H^{\cH} \rangle$ is stingray embedded in $G^{\BC
\BC'}$ by the construction of $\BC' \in \GL(d,q)$ in
\cref{ssec:GoingUpPhase2}. In this section the last $\steps$ \ref{SL5} to
\ref{SL7} are described in the next remark.

\begin{remark}\label{StepFiveToSevenOfSL}\ 
\begin{itemize}
\item[\SLFiv] Using $\cH$, construct elementary matrices $E_{j,n}(\prel_i)$
for $n < j \leq n'$.
\item[\SLSi] Using $\cH$, construct elementary matrices $E_{n,j}(\prel_i)$ for
$n < j \leq n'$.
\item[\SLSev] Using the elementary matrices of $\steps$ \ref{SL5} and \ref{SL6}
construct standard generators for $\langle H, H^{\cH} \rangle \cong \SL(n',q)$ by assembling permutation matrices corresponding to $n'$- and
$(n' - 1)$-cycles as in \cref{StandardGeneratorsSL}.
\end{itemize}
\end{remark}

In \ref{SL5} and \ref{SL6}, the elementary matrices $E_{j,n}(\prel_i)$ and
$E_{n,j}(\prel_i)$ are conjugated by $\cM$ and transformed by matrix
multiplications into transvections of $\SL(n',q) \cong \langle H,
H^{\cH} \rangle$ as stingray embedded elements of $G^{\BC \BC'}$. This
is necessary to subsequently compute permutation matrices of $\SL(n',q)$
in \ref{SL7}. For \ref{SL6} we need the transvections computed in \ref{SL5}.
Recall from \cref{transvection} the notation $T_{v,w}$ for
transvections.

The purpose of \ref{SL5} is to construct the elements $E_{j,n}(\prel_i)$
for $1 \leq i \leq f$ and $n+1 \leq j \leq n'$ as elements of $G^{\BC
\BC'}$. For this consider the transvections $T_{\prel_i
v_j',v_n'}^{\cH}$ for $1\leq i \leq f$ and $1 \leq j \leq n-1$.
One readily verifies that the
conjugate $T_{\prel_i v_j' , v_n' }^{\cH}$ is equal to
$T_{\prel_i v_j' \cH^{-\tr} , v_n' }$, since $v_n' \cH = v_n'$. Note
that the transvections $T_{\prel_i v_j' ,v_n' }^{\cH}$ fix $v_n'$ ,
since $\cH$ and the transvections $T_{\prel_i v_j' ,v_n' }$ fix $v_n'$.
Together this implies that $\langle v_n' \mid v_j' \cH^{-\tr} \rangle = 0$
holds for $1 \leq j \leq n - 1$.
Thus, all vectors in
\[ A := \{\prel_i v_j'
\mid 1 \leq i \leq f, 1 \leq j \leq n-1 \} \cup \{\prel_i v_j'
\cH^{-\tr} \mid 1 \leq i \leq f, 1 \leq j \leq n-1 \} \]
are orthogonal to $v_n'$ with respect to the standard scalar product, and we
know an MSLP for $T_{w,v_n' }$ for all vectors $w \in A$. By \ref{C3} we
know that the vectors in $A$ span the subspace $\langle v_1' ,\dotsc
,v_{n-1}', v_{n+1}' ,\dotsc ,v_{n'}'\rangle$. 
We can thus express any $\prel_i v_j'$ with $1 \leq i \leq f$ and $n+1 \leq j \leq n'$
as an $\F_p$-linear combination $\sum_{a\in A}\mu_a a$,
and from this get $T_{\prel_i v_j',v_n'} = \prod_{a\in A} T_{a,v_n'}^{\mu_a}$.
Hence we can now
construct the elements $E_{j,n}(\prel_i)$ as elements of $G^{\BC\BC'}$.

\begin{remark}\ 
\begin{enumerate}
\item Note that we have to perform linear algebra computations over
the prime field $\F_p$ here, but once again we use the results of the
computations only to write an MSLP for these transvections.
\item In an implementation one would check \ref{C3} already in
\algoref{ComputeStrongDoublingElement} to avoid unnecessary computations at
this point.
\end{enumerate}
\end{remark}

The idea of \ref{SL5} is compute $E_{j,n}(\prel_i)^{\cH}$ for $1 \leq i
\leq f$ and $1 \leq j \leq n-1$ and somehow to construct
$E_{j,n}(\prel_i)$ for $1 \leq i \leq f$ and $n+1 \leq j \leq n'$ using
the elements $\{E_{j,n}(\prel_i)^{\cH} \mid 1 \leq i \leq f, 1 \leq j
\leq n-1 \}$. To obtain a usable algorithm we start by computing
$T_{\prel_i v_j',v_n'}^{\cH}$ as elements of $G^{\BC \BC'}$, i.e.\
$E_{j,n}(\prel_i)^{\cH}$, which is given in \cref{SLStep5Representation}.

\begin{lemma}\label{SLStep5Representation}
Let $n'= \min\{2n - 1, d \}$, and let $\cH \in G^{\BC \BC'}$ be as
constructed in the previous $\steps$. Let
\[ a := \begin{pmatrix}
   1 & 0 & \hdots & 0 & \prel_i (\cH^{-1})_{1,j}&0 &  \hdots & 0\\
   0 & 1 & \hdots & 0 &\prel_i (\cH^{-1})_{2,j}  &0 &  \hdots & 0\\
   \vdots & \vdots & \ddots & \vdots & \vdots & \vdots & \ddots & \vdots \\
   0 & 0 & \hdots &1 & \prel_i (\cH^{-1})_{n-1,j} &0 &  \hdots & 0\\
   0 & 0 & \hdots & 0 & 1 & 0 & \hdots & 0\\
   0 & 0 & \hdots & 0 & \prel_i (\cH^{-1})_{n+1,j}  & 1 & \hdots & 0\\
   \vdots & \vdots & \ddots & \vdots & \vdots & \vdots & \ddots & \vdots \\
   0 & 0 & \hdots &0 & \prel_i (\cH^{-1})_{n',j}  &0 &  \hdots & 1\\
 \end{pmatrix} \in \SL(n',q).
\]
Then $E_{j,n}(\prel_i)^{\cH} = \diag(a,I_{d-n'}) \in G^{\BC \BC'}$.
\end{lemma}

\begin{proof}
If $k \in \{n'+1,\dotsc,d\}$, then
\[ e_k E_{j,n}(\prel_i)^{\cH}
= e_k \cH^{-1} E_{j,n}(\prel_i) \cH
= e_k E_{j,n}(\prel_i) \cH
= e_k \cH = e_k. \]
Moreover for $k \in \{1,\dotsc,n' \} - \{ n \}$
\begin{align*}
e_k \cH^{-1} E_{j,n}(\prel_i) \cH &= (\cH^{-1})_{k,-} E_{j,n}(\prel_i) \cH \\
&= ((\cH^{-1})_{k,-} + \prel_i (\cH^{-1})_{k,j} e_n) \cH \\
&= e_k \cH^{-1} \cH + \prel_i (\cH^{-1})_{k,j} e_n \cH
 = e_k + \prel_i (\cH^{-1})_{k,j} e_n
\end{align*}
and
\[ e_n \cH^{-1} E_{j,n}(\prel_i) \cH = e_n E_{j,n}(\prel_i) \cH = e_n \cH = e_n.
\qedhere \]
\end{proof}

\begin{remark}\label{Step5FurtherStepsSL}
Using the standard generators of $H \cong \SL(n,q)$, the matrix
$E_{j,n}(\prel_i)^{\cH}$ can be multiplied by transvections of $H$
resulting in row and column operations. Using row and column operations
another element of $\langle H, H^{\cH} \rangle$ can be constructed where
the entries in the $n$-th column of $E_{j,n}(\prel_i)^{\cH}$ above the
$(n,n)$ entry are zero such that matrices of the form
$\diag(a_{j,n,\prel_i}, I_{d-n'})$ are constructed where
\[
a_{j,n,\prel_i} := \begin{pmatrix}
   1 & 0 & \hdots & 0 & 0&0 &  \hdots & 0\\
   0 & 1 & \hdots & 0 &0 &0 &  \hdots & 0\\
   \vdots & \vdots & \ddots & \vdots & \vdots & \vdots & \ddots & \vdots \\
   0 & 0 & \hdots &1 & 0 &0 &  \hdots & 0\\
   0 & 0 & \hdots & 0 & 1 & 0 & \hdots & 0\\
   0 & 0 & \hdots & 0 & \prel_i (\cH^{-1})_{n+1,j}  & 1 & \hdots & 0\\
   \vdots & \vdots & \ddots & \vdots & \vdots & \vdots & \ddots & \vdots \\
   0 & 0 & \hdots &0 & \prel_i (\cH^{-1})_{n',j}  &0 &  \hdots & 1\\
 \end{pmatrix}.
\]
In the following we denote the process of using row and column
operations to construct another matrix with zeros at specific positions
as ``eliminating'' these entries. Using \ref{C3} the column vectors
$((\cH^{-1})_{n+1,i},\dotsc,(\cH^{-1})_{n',i})$ for $1 \leq i \leq n-1$ of
length $(n-1)$ below the $n$-row entry of $\cH^{-1}$ generate the full
$f(n-1)$ dimensional $\F_p$-vector space. Observe that the group $G_{n'}
:= \langle \{ a_{j,n,\prel_i} \mid 1 \leq j \leq n-1, 1 \leq i \leq f \}
\rangle$ is abelian and in fact isomorphic to $\F_q^{n-1}$ via the
isomorphism given by
\[
\varphi \colon \F_q^{n-1} \to G_{n'},\ (\lambda_{n+1}, \dotsc, \lambda_{n'}) \mapsto
T_{\sum_{n+1}^{n'} \lambda_i v_{i}, v_n} =
\begin{pmatrix}
   1 & 0 & \hdots & 0 & 0&0 &  \hdots & 0\\
   0 & 1 & \hdots & 0 &0 &0 &  \hdots & 0\\
   \vdots & \vdots & \ddots & \vdots & \vdots & \vdots & \ddots & \vdots \\
   0 & 0 & \hdots &1 & 0 &0 &  \hdots & 0\\
   0 & 0 & \hdots & 0 & 1 & 0 & \hdots & 0\\
   0 & 0 & \hdots & 0 & \lambda_{n+1} & 1 & \hdots & 0\\
   \vdots & \vdots & \ddots & \vdots & \vdots & \vdots & \ddots & \vdots \\
   0 & 0 & \hdots &0 & \lambda_{n'}  &0 &  \hdots & 1\\
 \end{pmatrix}.
\]
Note that our goal is the computation of the
transvections $E_{j,n}(\prel_i)$ for $1 \leq i \leq f$ and $n+1 \leq j
\leq n'$. 
By multiplying the matrices $a_{j,n,\prel_i}$ appropriately, the
standard basis of $\F_q^{n-1}$ as an $\F_p$-vector space, i.e.\ $\prel_1
e_1,\dotsc, \prel_f e_1, \dotsc,\prel_1 e_{n-1},\dotsc, \prel_f
e_{n-1}$, can be computed and thus the transvections $E_{j,n}(\prel_i)$
for $1 \leq i \leq f$ and $n+1 \leq j \leq n'$.
\end{remark}

We combine \cref{SLStep5Representation} and
\cref{Step5FurtherStepsSL} into \algoref{ComputeVerticalTransvections}.

\begin{algorithm}[H]
\caption{ComputeVerticalTransvections}
\label{ComputeVerticalTransvections}
\small{
{\nlnonumber
\IOKw{
\begin{tabular}{@{}ll}
        \KwIn{}& {\InputT} {\normalfont $\SL(n,q) \cong \MyIn{ \langle Y_n \rangle = H} \leq \SL(d,q)$ stingray embedded and constructively recognised}\\
        &{\InputT} {\normalfont $\MyIn{\cH} \in \SL(d,q)$ a strong doubling element}\\
        \KwOut{}& {\OutputT} {\normalfont $\MyOut{T_V} := \{ E_{j,n}(\prel_i) \mid 1 \leq i \leq f, n+1 \leq j \leq n' \} \subset \langle H, H^{\cH} \rangle$ transvections of $\SL(n',q)$} \\
        & {\OutputT} {\normalfont An MSLP $\MyOut{\slp}$ from $\MyIn{Y_n \cup \{ \cH \}}$ to $\MyOut{T_V}$} \\
\end{tabular}
} \\
}

\nlnonumber
\LinesNumbered
\setcounter{AlgoLine}{0}
\Begin($\FuncSty{ComputeVerticalTransvections} {(} \MyInT{H,\cH} {)}$)
{
   \tcp{\scriptsize{Function implements \ref{SL5} based on \cref{SLStep5Representation} and \cref{Step5FurtherStepsSL}}}
   $\tilde{T_V} \asgn [\,]$ \;
   \For{$\lambda \in \{\prel_1,\dotsc, \prel_f \}$ and $j \in \{2, \dotsc, n-1 \}$}
   {
       $T \asgn E_{j,n}(\lambda)^{\cH}$ \;
       \For{$k \in [1, \dotsc, n-1]$}
       {
        $T \asgn E_{k,n}(- \lambda (\cH^{-1})_{k,j}) T$ \tcp*[r]{\scriptsize{Note that $E_{k,n}(- \lambda (\cH^{-1})_{k,j}) \in H$}}
        }
        $\FuncSty{Add}(\tilde{T_V},T)$
   }
   $T_V \asgn [\,]$ \;
   \For{$\lambda \in \{\prel_1,\dotsc, \prel_f \}$ and $j \in \{n+1, \dotsc, n' \}$}
   {
   	$E_{j,n}(\lambda) \asgn $ Multiply the matrices of $\tilde{T_V}$ suitable \;
   	$\FuncSty{Add}(T_V,E_{j,n}(\lambda))$
   }
   $\slp \asgn$ MSLP for the computations of $T_V$ \;
   \Return $\MyOutT{(T_V,\slp)}$ \;
}
}
\end{algorithm}

In \ref{SL6}, the transvections $E_{n,j}(\prel_i)$ for $1 \leq i \leq f$
and $n+1 \leq j \leq n'$ as elements of $G^{\BC \BC'}$ should be
computed. This is more complicated than computing the transvections
$E_{j,n}(\prel_i)$ for $1 \leq i \leq f$ and $n+1 \leq j \leq n'$ in
$\step$ \ref{SL5} because $\cH^{-\tr}$ does not necessarily fix the
$n$-th basis vector anymore. The next lemma therefore carries out a
computation similar to the one described in \cref{SLStep5Representation} and has the desired consequences. Recall
that for a matrix $a \in \GL(d,q)$ the $i$-th row of $a$ is denoted by
$a_{i,-}$ and the $i$-th column of $a$ is denoted by $a_{-,i}$.

\begin{lemma}\label{SLStep6Representation}
Let $n'= \min\{2n - 1, d \}$ and let $\cH \in G^{\BC \BC'}$ be a strong
doubling element. Let $a \in \SL(n',q)$ with
\begin{gather*}
a := I_{n'} + \diag(\prel_i (\cH^{-1})_{1,n},\dotsc,\prel_i
(\cH^{-1})_{n-1,n},\prel_i,\prel_i (\cH^{-1})_{n+1,n},\dotsc,\prel_i
(\cH^{-1})_{n',n}) \\ \cdot (\cH_{j,-},\dotsc,\cH_{j,-})^{\tr}.
\end{gather*}
Then $E_{n,j}(\prel_i)^{\cH} = \diag(a,I_{d-n'}) \in G^{\BC \BC'}$.
\end{lemma}

\begin{proof}
If $k \in \{n'+1,\dotsc,d\}$, then
\[ e_k E_{n,j}(\prel_i)^{\cH} = e_k \cH^{-1} E_{n,j}(\prel_i) \cH = e_k
E_{n,j}(\prel_i) \cH = e_k \cH = e_k. \]
For $k \in \{1,\dotsc,n' \} - \{ n \}$
\begin{align*}
e_k \cH^{-1} E_{n,j}(\prel_i) \cH &= (\cH^{-1})_{k,-} E_{n,j}(\prel_i) \cH \\
&= ((\cH^{-1})_{k,-} + \prel_i (\cH^{-1})_{k,n} e_j) \cH \\
&= e_k \cH^{-1} \cH + \prel_i (\cH^{-1})_{k,n} e_j \cH
 = e_k + \prel_i (\cH^{-1})_{k,n} \cH_{j,-}
\end{align*}
and
\[ e_n \cH^{-1} E_{n,j}(\prel_i) \cH
= e_n E_{n,j}(\prel_i) \cH
= (\prel_i e_j + e_n) \cH
= \prel_i e_j \cH + e_n \cH
= \prel_i \cH_{j,-} + e_n.
\qedhere
\]
\end{proof}

After conjugating $E_{n,j}(\prel_i)$ by $\cH$ the result becomes
slightly more complicated as the conjugation of $E_{j,n}(\prel_i)$ by
$\cH$. \Cref{Step6Graphical} displays how the transvections $E_{n,j}(1)$
for $n+1 \leq j \leq n'$ can still be computed starting from
$E_{n,j}(\prel_i)^{\cH}$.

\begin{remark}\label{Step6Graphical}
Let $1 \leq j \leq n-1$, let $n'= \min\{2n - 1, d \}$ and let $\cH$ and
$\BC'$ be as constructed in the previous $\steps$. Moreover, we assume
that \ref{SL5} has already been performed, i.e.\ that MSLPs for the
transvections $E_{j,n}(\prel_i)$ for $1 \leq i \leq f$ and $n+1 \leq j
\leq n'$ have been computed.

As in \ref{SL5} the $n' \times n'$ top left block of $E_{n,j}(1)^{\cH}$
for $1 \leq j \leq n-1$ should be transformed into the transvection
$E_{n,n+j}(1)$. To illustrate the operations we are performing we
represent the $n' \times n'$ top left block of $E_{n,j}(1)^{\cH}$ as in
(\ref{Step6Graphical1}).
\begin{equation}
\label{Step6Graphical1}
\scriptsize\arraycolsep=2pt
\left(
  \begin{array}{ccc|c|ccc}
  \cR{\textbf{*}} & \cR{\hdots} & \cR{\textbf{*}} & \cGG{\textbf{*}} & \cO{\textbf{*}} & \cO{\hdots} & \cO{\textbf{*}} \\
  \cR{\vdots} & \cR{\ddots} & \cR{\vdots} & \cGG{\vdots} & \cO{\vdots} & \cO{\ddots} & \cO{\vdots} \\
  \cR{\textbf{*}} & \cR{\hdots} & \cR{\textbf{*}} & \cGG{\textbf{*}} & \cO{\textbf{*}} & \cO{\hdots} & \cO{\textbf{*}} \\
    \hline
  \cB{\textbf{*}} & \cB{\hdots} & \cB{\textbf{*}} & \textbf{*} & \cB{\textbf{*}} & \cB{\hdots} & \cB{\textbf{*}} \\
    \hline
  \cP{\textbf{*}} & \cP{\hdots} & \cP{\textbf{*}} & \cGG{\textbf{*}} & \cYY{\textbf{*}} & \cYY{\hdots} & \cYY{\textbf{*}} \\
  \cP{\vdots} & \cP{\ddots} & \cP{\vdots} & \cGG{\vdots} & \cYY{\vdots} & \cYY{\ddots} & \cYY{\vdots} \\
  \cP{\textbf{*}} & \cP{\hdots} & \cP{\textbf{*}} & \cGG{\textbf{*}} & \cYY{\textbf{*}} & \cYY{\hdots} & \cYY{\textbf{*}}
\end{array} \right).
\end{equation}
We are using $n' -n$ to deal with both cases for $n'$, i.e.\ $n' = 2n -
1 \leq d$ and $n' = d \neq 2n - 1$ simultaneously. If $n' = 2n - 1 \leq
d$, then $n' - n = 2n-1 - n = n-1$ and if $n' = d \neq 2n - 1$, then $n'
- n < n-1$. Remember that we have standard generators and, therefore,
all transvections of $\SL(n,q)$ for the top left $n \times n$ block
(consisting of the red, green, blue and black stars). We start by
eliminating the red and orange blocks by adding the $n$-th row using row
operations. Let $k \in \{1,\dotsc,n-1\}$ and $i$ as in
\cref{SLStep6Representation}. Then
\[
\begin{aligned}
(e_k + \prel_i (\cH^{-1})_{k,n} \cH_{j,-})
    + (-(\cH^{-1})_{k,n}) (\prel_i \cH_{j,-} + e_n)
&= e_k + \prel_i (\cH^{-1})_{k,n} \cH_{j,-}
       - (\cH^{-1})_{k,n} \prel_i \cH_{j,-}
       - (\cH^{-1})_{k,n} e_n \\
&= e_k - (\cH^{-1})_{k,n} e_n
\end{aligned}
\]
where $e_k + \prel_i (\cH^{-1})_{k,n} \cH_{j,-}$ is the $k$-th row of
$E_{n,j}(1)^{\cH}$ and $\prel_i \cH_{j,-} + e_n$ is the $n$-th row of
$E_{n,j}(1)^{\cH}$. Hence, the addition of these rows can be achieved by
multiplying with $E_{k,n}(- (\cH^{-1})_{k,n})$ from the left. Note that
$E_{k,n}(- (\cH^{-1})_{k,n}) \in \SL(n,q)$ and, therefore, we can write
$E_{k,n}(- (\cH^{-1})_{k,n})$ as a word in $Y_n$. After these $n-1$ row
operations, the matrix of (\ref{Step6Graphical1}) is transformed into
the matrix of (\ref{Step6Graphical2}):
\begin{equation}
\label{Step6Graphical2}
\scriptsize\arraycolsep=2pt
\left(
  \begin{array}{ccc|c|ccc}
 & &  & \cGG{\textbf{*}} & &  &  \\
   & I_{n-1} &  & \cGG{\vdots} &  & 0 &  \\
  & &  & \cGG{\textbf{*}} & &  &  \\
  \hline
  \cB{\textbf{*}} & \cB{\hdots} & \cB{\textbf{*}} & \textbf{*} & \cB{\textbf{*}} & \cB{\hdots} & \cB{\textbf{*}} \\
   \hline
  \cP{\textbf{*}} & \cP{\hdots} & \cP{\textbf{*}} & \cGG{\textbf{*}} & \cYY{\textbf{*}} & \cYY{\hdots} & \cYY{\textbf{*}} \\
  \cP{\vdots} & \cP{\ddots} & \cP{\vdots} & \cGG{\vdots} & \cYY{\vdots} & \cYY{\ddots} & \cYY{\vdots} \\
  \cP{\textbf{*}} & \cP{\hdots} & \cP{\textbf{*}} & \cGG{\textbf{*}} & \cYY{\textbf{*}} & \cYY{\hdots} & \cYY{\textbf{*}}
\end{array} \right).
\end{equation}
We proceed analogously with the rows below the $n$-th row using the same
argument for $k \in \{n+1,\dotsc,n'\}$. Notice that this is only
possible since we have already computed the transvections
$E_{j,n}(\prel_i)$ for $1 \leq i \leq f$ and $n+1 \leq j \leq n'$ in
\ref{SL5}. After these additional $n-1$ row operations we transform
(\ref{Step6Graphical2}) into (\ref{Step6Graphical3}).
\begin{equation}
\label{Step6Graphical3}
\scriptsize\arraycolsep=2pt
\left(
  \begin{array}{ccc|c|ccc}
 & &  & \cGG{\textbf{*}} & &  &  \\
   & I_{n-1} &  & \cGG{\vdots} &  & 0 &  \\
  & &  & \cGG{\textbf{*}} & &  &  \\
  \hline
  \cB{\textbf{*}} & \cB{\hdots} & \cB{\textbf{*}} & \textbf{*} & \cB{\textbf{*}} & \cB{\hdots} & \cB{\textbf{*}} \\
   \hline
 & & & \cGG{\textbf{*}} & &  & \\
   & 0& \ & \cGG{\vdots} &  & I_{n-1} & \\
   & &  & \cGG{\textbf{*}} &  &  &
\end{array} \right).
\end{equation}
Notice that the green entries are known, i.e.\ the entry at position
$(k,n)$ is $- (\cH^{-1})_{k,n}$, since $(e_k + \prel_i (\cH^{-1})_{k,n}
\cH_{j,-}) + (- (\cH^{-1})_{k,n}) (\prel_i \cH_{j,-} + e_n) = e_k -
(\cH^{-1})_{k,n} e_n$. By adding the columns 1 to $n-1$ multiplying by
the corresponding scalars to the $n$-th column, the green entries in the
$n$-th column above the $n$-th entry can be eliminated. This can be
performed by multiplying the matrices $E_{k,n}((\cH^{-1})_{k,n}) \in H$
from the right. The entry in position $(n,n)$ is changed to $1$ during
these operations if this was not the case which is shown at the end of
this remark. After these column operations (\ref{Step6Graphical3}) is
transformed into (\ref{Step6Graphical4}).
\begin{equation}
\label{Step6Graphical4}
\scriptsize\arraycolsep=2pt
\left(
  \begin{array}{ccc|c|ccc}
 & &  & 0& &  &  \\
   & I_{n-1} &  & \vdots &  & 0 &  \\
  & &  & 0 & &  &  \\
  \hline
  \cB{\textbf{*}} & \cB{\hdots} & \cB{\textbf{*}} & \textbf{*} & \cB{\textbf{*}} & \cB{\hdots} & \cB{\textbf{*}} \\
   \hline
 & & & \cGG{\textbf{*}} & &  & \\
   & 0& \ & \cGG{\vdots} &  & I_{n-1} & \\
   & &  & \cGG{\textbf{*}} &  &  &
\end{array} \right).
\end{equation}
The rest is now trivial. The rows 1 to $n-1$ are added to the $n$-th row
in order to eliminate the entries to the left of the $n$-th entry in the
$n$-th row. Moreover, the columns $n+1$ to $n'$ are added to the $n$-th
column to eliminate the entries below the $n$-th entry, resulting in the
final matrix given in (\ref{Step6Graphical5}).
\begin{equation}
\label{Step6Graphical5}
\scriptsize\arraycolsep=2pt
\left(
  \begin{array}{ccc|c|ccc}
 & &  & 0& &  &  \\
   & I_{n-1} &  & \vdots &  & 0 &  \\
  & &  & 0 & &  &  \\
  \hline
  0 & \hdots & 0 & 1 & \cB{\textbf{*}} & \cB{\hdots} & \cB{\textbf{*}} \\
   \hline
 & & & 0 & &  & \\
   & 0& \ & \vdots &  & I_{n-1} & \\
   & &  & 0 &  &  &
\end{array} \right).
\end{equation}
The entry in Position $(n,n)$ of (\ref{Step6Graphical5}) must be 1 since
this is an upper triangular matrix which is also contained in $\SL$.
Note that the entries in position $(n,n)$ in this matrix and matrix in
(\ref{Step6Graphical4}) are the same which is why it already was $1$
before these final row and column operations. Moreover, the entries in
blue are equal to $(\cH_{j,-})_{n+1,\dotsc,n'}$ which are equal to
$e_1,\dotsc,e_{n-1}$ by choosing the basis $\MB'$ of $\step$ \ref{SL4}
advisedly. Overall, $E_{n,j}(1)^{\cH}$ is transformed into
$E_{n,n+j}(1)$ using only transvections for which an MSLP is known.
Therefore, an MSLP evaluating to $E_{n,n+j}(1)$ is constructed.
\end{remark}

We also give a pseudo code performing \ref{SL6} as \algoref{ComputeHorizontalTransvections} in 
\cref{ComputeHorizontalTransvections}.

\begin{algorithm}[H]
\caption{ComputeHorizontalTransvections}
\label{ComputeHorizontalTransvections}
\small{
{\nlnonumber
\IOKw{
\begin{tabular}{@{}ll}
        \KwIn{}& {\InputT} {\normalfont $\SL(n,q) \cong \MyIn{ \langle Y_n \rangle = H} \leq \SL(d,q)$ stingray embedded and constructively recognised}\\
        &{\InputT} {\normalfont $\MyIn{\cH} \in \SL(d,q)$ a strong doubling element}\\
        \KwOut{}& {\OutputT} {\normalfont $\MyOut{T_H} := \{ E_{n,j}(1) \mid n+1 \leq j \leq n' \} \subset \langle H, H^{\cH} \rangle$ transvections of $\SL(n',q)$} \\
        & {\OutputT} {\normalfont An MSLP $\MyOut{\slp}$ from $\MyIn{Y_n \cup \{ \cH \}} \cup T_V$ to $\MyOut{T_H}$} \\
\end{tabular}
} \\
}

\nlnonumber
\LinesNumbered
\setcounter{AlgoLine}{0}
\Begin($\FuncSty{ComputeHorizontalTransvections} {(} \MyInT{H,\cH} {)}$)
{
   \tcp{\scriptsize{Function implements \ref{SL6} based on \cref{SLStep6Representation} and \cref{Step6Graphical}}}
   $T_H \asgn [\,]$ \;
   \For{$j \in \{2, \dotsc, n-1 \}$}
   {
       $T \asgn E_{n,j}(1)^{\cH}$ \;
       \For{$k \in [1, \dotsc, n-1]$}
       {
        $T \asgn E_{k,n}(- (\cH^{-1})_{k,n}) T $ \tcp*[r]{\scriptsize{As in (\ref{Step6Graphical2}) of \cref{Step6Graphical}}}
        }
        \For{$k \in [n+1, \dotsc, n']$}
       {
        $T \asgn E_{k,n}(- (\cH^{-1})_{k,n})  T $ \tcp*[r]{\scriptsize{As in (\ref{Step6Graphical3}) of \cref{Step6Graphical}}}
        }
        \For{$k \in [1, \dotsc, n-1]$}
       {
        $T \asgn T E_{k,n}((\cH^{-1})_{k,n}) $ \tcp*[r]{\scriptsize{As in (\ref{Step6Graphical4}) of \cref{Step6Graphical}}}
        }
        \For{$k \in [1, \dotsc, n-1]$}
       {
        $T \asgn E_{n,k}(- (\cH^{-1})_{j,k}) T $ \tcp*[r]{\scriptsize{As in (\ref{Step6Graphical5}) of \cref{Step6Graphical}}}
        }
         \For{$k \in [n+1, \dotsc, n']$}
       {
        $T \asgn T E_{k,n}((\cH^{-1})_{k,n})$ \tcp*[r]{\scriptsize{As in (\ref{Step6Graphical5}) of \cref{Step6Graphical}}}
        }
        $\FuncSty{Add}(T_H,E)$
   }
   $\slp \asgn$ MSLP for the computations of $T_H$ \;
   \Return $\MyOutT{(T_H,\slp)}$ \;
}
}
\end{algorithm}

In the last and final $\step$, the $n'$ and $n'-1$ cycles are now
assembled. This can easily be realized with the transvections from
\ref{SL5} and \ref{SL6} as described in \cref{Step7SL}.

\begin{lemma}\label{Step7SL}
The permutation matrices $\PME',\PMZ'$ as in \cref{StandardGeneratorsSL}
for $\SL(n',q)$ can be computed using the matrices of the set $X =
\{\PME,\PMZ,E_{i,n}(1),E_{n,i}(1) \}$ for $n+1 \leq i \leq n'$.
\end{lemma}

\begin{proof}
Transpositions can be computed easily as
$E_{i,n}^{-1}(1)E_{n,i}(1)E_{i,n}^{-1}(1)$ is the permutation matrix
which corresponds to $(n,i) \in S_{n'}$ for $n < i \leq n'$, where the
entry in position $(i,n)$ is equal to $-1$ and
$E_{i,n}(1)E_{n,i}^{-1}(1)E_{i,n}(1)$ is the permutation matrix which
corresponds to $(n,i) \in S_{n'}$ for $n < i \leq n'$, where the entry
in position $(n,i)$ is equal to $-1$. Moreover
\begin{align*}
&(n,n') \cdot (n,n'-1) \cdot (n,n'-2) \cdot \dotsc \cdot (n,n+1) \\
=\, &(n,n',n'-1) \cdot (n,n'-2) \cdot \dotsc \cdot (n,n+1) \\
=\, &(n,n',n'-1,\dotsc,n+1).
\end{align*}
and
\begin{align*}
(n,n-1,\dotsc,1) \cdot (n,n',n'-1,\dotsc,n+1) &= (n,n-1,\dotsc,1,n',n'-1,\dotsc,n+1) \\
  &= (n',n'-1,\dotsc,1), \\
(n,n-1,\dotsc,2) \cdot (n,n',n'-1,\dotsc,n+1) &= (n,n-1,\dotsc,2,n',n'-1,\dotsc,n+1) \\
  &= (n',n'-1,\dotsc,2).
\end{align*}
Because of the position of $-1$ in the transpositions, the matrices
correspond to the standard generators of \cref{StandardGeneratorsSL}.
\end{proof}

\subsection{\GoingUp\ step} \label{ssec:GoingUpStep}
Finally all $\steps$ and subalgorithms of this section are combined into
one single algorithm which can be used for a single \GoingUp\ step as
stated in \cref{DoublingTheDimension}. The seven $\steps$ of the
previous sections are summarized in \cref{StepsOfGoingUpLA}. Recall that
$\steps$ \ref{SL1} to \ref{SL3} are formulated in \cref{StepOneToThreeOfSL}
of \cref{ssec:GoingUpPhase1}, that $\step$ \ref{SL4} is formulated in
\cref{StepFourOfSL} of \cref{ssec:GoingUpPhase2} and that $\steps$
\ref{SL5} to \ref{SL7} are formulated in \cref{StepFiveToSevenOfSL} of
\cref{ssec:GoingUpPhase3}.

\begin{remark}\label{StepsOfGoingUpLA}
The following seven $\steps$ must be performed for
\cref{DoublingTheDimension} and, therefore, for one \GoingUp\ step:
\begin{enumerate}[label={(SL\arabic*)},leftmargin=*]
\item\label{SL1} Construct an element $\GUE \in H$ which has a fixed space
of dimension $d - n + 1$.
\item\label{SL2} Choose random elements $\cE \in G^{\BC}$ until $\cT :=
\GUE^{\cE}$ is a weak doubling element.
\item\label{SL3} Find a conjugate $\cM \in G^{\BC}$ of $\cT$ which is a doubling element.
\item\label{SL4} Compute a base change matrix $\BC'$ such that $\langle
H,H^{\cM} \rangle^{\BC'}$ is stingray embedded in $G^{\BC \BC'}$. Set
$\cH := \cM^{\BC'}$ and verify whether $\cH$ is a strong doubling
element.
\item\label{SL5} Using $\cH$, construct transvections $E_{j,n}(\prel_i)$
for $n < j \leq n'$.
\item\label{SL6} Using $\cH$, construct transvections $E_{n,j}(\prel_i)$ for
$n < j \leq n'$.
\item\label{SL7} Using the transvections of \ref{SL5} and \ref{SL6} construct
standard generators for $\langle H, H^{\cH} \rangle \cong \SL(n',q)$ by
assembling permutation matrices corresponding to $n'$- and $(n' -
1)$-cycles as in \cref{StandardGeneratorsSL}.
\end{enumerate}
\end{remark}

Note that the condition \ref{C3} is tested by
\algoref{ComputeStrongDoublingElement} in \ref{SL4} and if $\cH$ does not
satisfy \ref{C3}, then we return to \ref{SL2}. An overall algorithm for one
\GoingUp\ step is given in pseudo code as \algoref{GoingUpStep}.

\begin{algorithm}[H]
\caption{GoingUpStep}
\label{GoingUpStep}
\small{
{\nlnonumber
\IOKw{
\begin{tabular}{@{}ll}
        \KwIn{}& {\InputT} {\normalfont $\MyIn{\langle X \rangle = G} \leq \SL(d,q)$}\\
        &{\InputT} {\normalfont A base change matrix $\MyIn{\BC} \in \GL(d,q)$}\\
        &{\InputT} {\normalfont $\SL(n,q) \cong \MyIn{ \langle Y_n \rangle = H} \leq \MyIn{G}^{\MyIn{\BC}}$ stingray embedded and constructively recognised}\\
        &{\InputT} {\normalfont $\MyIn{\MN} \in \N$}\\
        \KwOut{}& \MyFail{fail} OR $\MyOut{(Y_{n'},\BC',\slp,\MN')}$ where \\
        &{\OutputT} {\normalfont $\SL(\min\{2n-1,d\},q) \cong \MyOut{ \langle Y_{n'} \rangle = \tilde{H}}$,} \\
        & {\OutputT} {\normalfont $\MyOut{\BC'}$ is a base change matrix such that $\MyOut{\tilde{H}}$ is stingray embedded in $\MyIn{G}^{\MyIn{\BC}\MyOut{\BC'}}$,} \\
        & {\OutputT} {\normalfont $\MyOut{\slp}$ is an MSLP from $\MyIn{X \cup Y_n}$ to the standard generators $\MyOut{ Y_{n'} }$ of $\MyOut{\tilde{H}}$ and} \\
        &{\OutputT} {\normalfont $\MyOut{\MN'} \in \N$ where $\MyIn{\MN}-\MyOut{\MN'}$ is the number of random selections that were used} \\
\end{tabular}
} \\
}

\nlnonumber
\LinesNumbered
\setcounter{AlgoLine}{0}
\Begin($\FuncSty{GoingUpStep} {(} \MyInT{G,H,\BC, \MN} {)}$)
{
   $(\cH, \BC', \slp_1,\MN) \asgn \algoref{ComputeStrongDoublingElement}(G, H, \BC, \MN)$ \;
   \uIf{$\cH = $ fail}
   {
   	\Return $\MyOutTT{\fail}$ \;
   }
   $T_V, \slp_2 \asgn \algoref{ComputeVerticalTransvections}(H,\cH)$ \;
   $T_H, \slp_3 \asgn \algoref{ComputeHorizontalTransvections}(H, \cH)$ \;
   Use $T_V$ and $T_H$ to construct $\PME$ and $\PMZ$ of $\SL(n',q)$ as an MSLP $\slp_4$ using \cref{Step7SL} \tcp*[r]{\scriptsize{\ref{SL7}}}
   Compose $\slp_1, \slp_2, \slp_3, \slp_4$ into one MSLP $\slp$ \;
   \Return $\MyOutT{(\langle H, H^{\cH} \rangle,\BC',\slp,\MN)}$ \;
}
}
\end{algorithm}

\begin{theorem}\label{GoingUpStepAlgorithm}
\algoref{GoingUpStep} terminates using at most $\MN$ random selections and works correctly.
\end{theorem}

\begin{proof}
The correctness is clear since the correctness of each $\step$ has been
proven in the preceding sections. It remains to show termination.
Note that most lines are deterministic
and can be performed in finite time. The only critical line is line 1,
i.e. finding a suitable $\cH$, which is controlled by $\MN$ to have only
a finite number of tries.
\end{proof}

\subsection{Combining \GoingUp\ steps} \label{sec:GoingUp}
We have demonstrated how we can compute standard generators for a
stingray embedded special linear group of dimension nearly twice that of
a given stingray embedded special linear group with standard generators.
In this section \algoref{GoingUpStep} is used to develop an algorithm which
can be used to compute standard generators for $G \leq \SL(d,q)$ with $G
\cong \SL(d,q)$, where $H \leq G$ with $H \cong \SL(2,q)$ are given and
standard generators of $H$ are known.

\begin{algorithm}[H]
\caption{GoingUp}
\label{GoingUp}
\small{
{\nlnonumber
\IOKw{
\begin{tabular}{@{}ll}
        \KwIn{}& {\InputT} {\normalfont $\MyIn{\langle X \rangle = G} = \SL(d,q)$}\\
        &{\InputT} {\normalfont A base change matrix $\MyIn{\BC} \in \GL(d,q)$}\\
        &{\InputT} {\normalfont $\SL(2,q) \cong \MyIn{\langle Y_2 \rangle = H} \leq \MyIn{G}^{\MyIn{\BC}}$ stingray embedded and constructively recognised}\\
        &{\InputT} {\normalfont $\MyIn{\MN} \in \N$}\\
        \KwOut{}& \MyFail{fail} OR $\MyOut{(\BC',\slp,\MN')}$ where \\
        &{\OutputT} {\normalfont $\MyOut{\BC'}$ is a base change matrix,} \\
        &{\OutputT} {\normalfont $\MyOut{\slp}$ is an MSLP from $\MyIn{X \cup Y_2}$ to the standard generators of $\MyIn{G}^{\MyOut{\BC'}}$ and} \\
        &{\OutputT} {\normalfont $\MyOut{\MN'} \in \N$ where $\MyIn{\MN}-\MyOut{\MN'}$ is the number of random selections that were used} \\
\end{tabular}
} \\
}

\nlnonumber
\LinesNumbered
\setcounter{AlgoLine}{0}
\Begin($\FuncSty{GoingUp} {(} \MyInT{G,H,\BC,\MN} {)}$)
{
  $n \asgn 2$ \;
  $\slp \asgn $ an MSLP from $X \cup Y_2$ to $X \cup Y_2$\;
  \While{$ n < d $}
  {
  $n \asgn \min \{ 2\cdot n - 1, d \}$ \tcp*[r]{\scriptsize{Nearly Double the dimension.}}
  $(H, \BC, \slp',\MN) \asgn \algoref{GoingUpStep}(G,H,\BC, \MN)$ \tcp*[r]{\scriptsize{\cref{GlobalCounter,OutputCheck}}}
  $\slp \asgn$ Compose $\slp$ and $\slp'$ \;
  }

  \Return $\MyOutT{(\BC,\slp,\MN)}$
}
}
\end{algorithm}

\begin{theorem}
\algoref{GoingUp} terminates using at most $\MN$ random selections and works correctly.
\end{theorem}

\begin{proof}
Follows immediately from \cref{GoingUpStepAlgorithm}.
\end{proof}

\section{Implementation and timing}
The \GoingDown, \BaseCase\ and \GoingUp\ algorithms presented here have been
implemented in \GAP~\cite{GAP} by the fourth author. In this section we
compare the run-time of this implementation with the constructive recognition
algorithms \cite{CRODD,CREVEN} available in \Magma~\cite{Magma}.

\newcommand\NOPE{\text{timeout}}
\newcommand\OOPS[2]{\footnote{#1 runs terminated, #1 exceeded the time limit}}
\begin{table}[ht]
\begin{minipage}{\linewidth}
\[
\begin{array}{c@{\hskip 15mm}rrr@{\hskip 15mm}rrr@{\hskip 15mm}rrr}
\toprule
& \multicolumn{2}{c}{q=4} &
& \multicolumn{2}{c}{q=5} &
& \multicolumn{2}{c}{q=121}
 \\ \cmidrule(r){2-3} \cmidrule(r){5-6} \cmidrule(r){8-9}
d       &\GAP\ &\Magma\   &&\GAP\ &\Magma\ &&\GAP\  &\Magma\ \\
\midrule
 100    &  < 1 &     7    &&  < 1 &    3   &&   < 1 &   86 \\
 200    &    6 &    20    &&    1 &    8   &&     2 &  436 \\
 300    &   11 &    39    &&    4 &   16   &&     7 & 12595\mathclap{\hspace{2mm}{}^a} \\
 500    &   62 &   104    &&   16 &   62   &&    31 & \NOPE \\
 700    &  168 &   257    &&   40 &  249   &&    79 & - \\
1000    &  337 &   515    &&   81 & 16047\mathclap{\hspace{2mm} {}^b} &&   176 & -  \\
2000    & 2671 &  4014    &&  532 & \NOPE  &&  2220 & -  \\
3000    & 9626 & 14780    && 2028 & \NOPE  && 11637 & -  \\
\bottomrule
   \multicolumn{6}{l}{{}^{a}\text{\footnotesize{Four runs completed, six exceeded the time limit}}} \\
   \multicolumn{6}{l}{{}^{b}\text{\footnotesize{Two runs completed, eight exceeded the time limit}}} \\
\end{array}
\]
\end{minipage}
\caption{Time in seconds comparing our \GAP\ implementation
against \Magma. An entry with a dash means we did not measure this computation.
Entries with ``\NOPE'' indicate that no run completed within 20\;000 seconds.
}
\label{RuntimeSLComparision}
\end{table}

Since the algorithms are implemented in different computer algebra systems,
which have quite different performance profiles, a fair direct comparison is
difficult. Generally, the underpinning support functions for working with
matrices and polynomials over finite fields in \Magma\ outperform those in
\GAP. Despite this, the result of our comparison shown in
Table~\ref{RuntimeSLComparision} is quite favorable to our implementation.

All computations for both implementations were performed on a Linux server
with two AMD EPYC 9554 64-core processors and 1.5 TB RAM, running 
Gentoo 2.14 with kernel version 6.1.69.
For \Magma\ version 2.28-8 was used and for \GAP\ version 4.13.0 together
with its default set of packages loaded plus the \Pkg{cvec} package
version 2.8.1 (for faster characteristic polynomials)
with a modified version of \Pkg{recog} 1.4.2, where the modification is that
we inserted our new algorithms. These modifications will be part of a
future \Pkg{recog} release.

For \Magma\ we used
\verb|ClassicalConstructiveRecognition(SL(d,q),"SL",d,q);|
as test command while for GAP
\verb|RECOG.FindStdGensSmallerMatrices_SL(SL(d,q));| was used.
For each used parameter pair $(d,q)$ we run the command ten times in
each system. \Cref{RuntimeSLComparision} reports the average of these
runs. Note that computations that exceeded 20\,000 seconds (about 5.5
hours) were aborted. In some cases only a subset of the runs for a given
pair $(d,q)$ terminated within that time limit. In this case we treated
the aborted computations as if they had completed within 20\,000 seconds
(which is generous; in some cases we actually waited for 15 hours before
aborting).

\end{document}